\documentclass{amsart}

\usepackage{amssymb}
\usepackage{cite}      

\newtheorem{thm}          {Theorem}     [section]
\newtheorem{prop}         {Proposition} [section]
\newtheorem{lemma}        {Lemma}	[section]

\newtheorem{cor}          {Corollary}   [section]

\newtheorem{rem}          {Remark}	[section]

\newcommand{\RR}{\mathbb{R}}

\newcommand{\dist}{\operatorname{dist}}

\newcommand{\diam}{\operatorname{diam}}

\renewcommand{\div}{\operatorname{div}}

\newcommand{\e}{{\rm e}}

\newcommand{\eps}{\varepsilon}

\newcommand{\bs}{\boldsymbol}

\newcommand{\be}{
\begin{equation}
}
\newcommand{\bel}[1]{
\begin{equation}
\label{#1}}
\newcommand{\ee}{
\end{equation}
}
\newtheorem{subn}{\name}

\newcommand{\bsn}[1]{\def\name{#1}
\begin{subn}}
\newcommand{\esn}{
\end{subn}}
\newtheorem{sub}{\name}[section]
\newcommand{\es}{
\end{sub}}
\newcommand{\bsl}[1]{
\begin{sub}\label{#1}}
\newcommand{\bth}[1]{\def\name{Theorem}
\begin{sub}\label{t:#1}}
\newcommand{\blemma}[1]{\def\name{Lemma}
\begin{sub}\label{l:#1}}
\newcommand{\bcor}[1]{\def\name{Corollary}
\begin{sub}\label{c:#1}}
\newcommand{\bdef}[1]{\def\name{Definition}
\begin{sub}\label{d:#1}}
\newcommand{\bprop}[1]{\def\name{Proposition}
\begin{sub}\label{p:#1}}

\newcommand{\BA}{
\begin{array}}
\newcommand{\EA}{
\end{array}}
\newcommand{\BAN}{\renewcommand{\arraystretch}{1.2}
\setlength{\arraycolsep}{2pt}
\begin{array}}
\newcommand{\BAV}[2]{\renewcommand{\arraystretch}{#1}
\setlength{\arraycolsep}{#2}
\begin{array}}
\newcommand{\BSA}{
\begin{subarray}}
\newcommand{\ESA}{
\end{subarray}}
\newcommand{\BAL}{
\begin{aligned}}
\newcommand{\EAL}{
\end{aligned}}
\newcommand{\BALG}{
\begin{alignat}}
\newcommand{\EALG}{
\end{alignat}}
\newcommand{\BALGN}{
\begin{alignat*}}
\newcommand{\EALGN}{
\end{alignat*}}
\newcommand{\note}[1]{\textit{#1.}\hspace{2mm}}

\newcommand{\Remark}{\note{Remark}}

\newcommand{\abs}[1]{\left |#1\right |}
\newcommand{\norm}[1]{\left \|#1\right \|}

\def\angb<#1>{\langle #1 \rangle}

\newcommand{\myfrac}[2]{{\displaystyle \frac{#1}{#2} }}
\newcommand{\myint}[2]{{\displaystyle \int_{#1}^{#2}}}

\newcommand{\prt}{
\partial}

\def\ga{\alpha}     \def\gb{\beta}       
\def\gth{\theta}                         
\def\gf{\phi}           
            \def\gl{\lambda}
        \def\gn{\nu}         
            
\def\gs{\sigma}       \def\gt{\tau}
      
                \def\gz{\zeta}
\def\Gg{\Gamma}     \def\Gd{\Delta}      \def\Gf{\Phi}

\def\Gw{\Omega}              

   \def\CB{{\mathcal B}}

\def\CJ{{\mathcal J}}

   \def\BBR {\mathbb R}

\numberwithin{equation}{section}

\begin{document}

\title{Isolated boundary singularities of semilinear 
 elliptic equations}

\author{Marie-Fran\c{c}oise Bidaut-V\'eron}
\address{
Universit\'e Fran\c cois Rabelais\hfill\break\indent
Laboratoire de Math\'ematiques et Physique Th\'eorique (UMR CNRS 6083)\hfill\break\indent F\'ed\'eration Denis Poisson\hfill\break\indent
37200 Tours, France} 

\author{Augusto C.\@ Ponce}
\address{
Universit\'e catholique de Louvain\hfill\break\indent 
D\'epartement de math\'ematique\hfill\break\indent 
Chemin du Cyclotron 2\hfill\break\indent 
1348 Louvain-la-Neuve, Belgium} 

\author{Laurent V\'eron}
\address{
Universit\'e Fran\c cois Rabelais\hfill\break\indent
Laboratoire de Math\'ematiques et Physique Th\'eorique (UMR CNRS 6083)\hfill\break\indent F\'ed\'eration Denis Poisson\hfill\break\indent
37200 Tours, France} 

\date{\today}

\begin{abstract}
Given a smooth domain $\Omega\subset\RR^N$ such that $0 \in \partial\Omega$ and given a nonnegative smooth function $\zeta$ on $\partial\Omega$, we study the behavior near $0$ of positive solutions of $-\Delta u=u^q$ in $\Omega$ such that $u = \zeta$ on $\partial\Omega\setminus\{0\}$. We prove that if $\frac{N+1}{N-1} < q < \frac{N+2}{N-2}$, then $u(x)\leq C \abs{x}^{-\frac{2}{q-1}}$ and we compute the limit of $\abs{x}^{\frac{2}{q-1}} u(x)$ as $x \to 0$. We also investigate the case $q= \frac{N+1}{N-1}$. The proofs rely on the existence and uniqueness of solutions of related equations on spherical domains.
\end{abstract}

\maketitle
\tableofcontents


\section{Introduction and main results}\label{sec1}

Let $\Omega$ be a smooth open subset of $\RR^N$, with $N \geq 2$, such that
$0 \in \partial\Omega$. Given $q>1$ and $\zeta \in C^\infty(\partial\Omega)$ with $\zeta \ge 0$ on $\partial\Omega$, consider the problem
\begin{equation}\label{E1}
\left\{
\begin{alignedat}{2}
-\Delta u & = u^q && \quad\mbox{in } \Omega,\\
u & \geq 0 && \quad\mbox{in } \Omega,\\
u & = \zeta && \quad\mbox{on } \partial\Omega\setminus\{0\}.
\end{alignedat}
\right.
\end{equation}  
By a solution of \eqref{E1} we mean a function $u\in C^2(\Omega) \cap
C(\overline\Omega\setminus\{0\})$ which satisfies \eqref{E1} in the classical sense.
A solution may develop an isolated singularity at $0$. Our main goal in this paper is to describe the behavior of $u$ in a neighborhood of this point.

\medskip
In the study of boundary singularities of \eqref{E1}, one finds three critical exponents; namely,
\begin{equation*}
q_1= \tfrac{N+1}{N-1}, \quad  q_2= \tfrac{N+2}{N-2} \quad \text{and} \quad q_3= \tfrac{N+1}{N-3},
\end{equation*}
with the usual convention if $N=2$ or $N=3$. When $1<q<q_1$, it is proved by Bidaut-V\'eron--Vivier~\cite{BidViv:00} that for every solution
$u$ of \eqref{E1} there exists $\alpha \geq 0$ (depending on $u$) such that 
\begin{equation*}
u(x)=\alpha \abs x^{-N} \dist(x,\partial\Omega) \, \big(1+o(1) \big)\quad \mbox{as } x\to 0.
\end{equation*}
In this paper we mainly investigate the case $q_1 \le q < q_3$. 

\medskip
The counterpart of \eqref{E1} for an interior singularity,
\[
-\Gd u = u^{q} \quad \text{in } \Omega\setminus\{x_0\},
\]
where $x_0 \in \Omega$, was studied by P.-L.~Lions~\cite{Lio:80} in the subcritical case $1<q < \frac{N}{N-2}$, by Aviles~\cite{Avi:87} when $q = \frac{N}{N-2}$ and by Gidas-Spruck~\cite{GidSpr:81} in the range  $\frac{N}{N-2} < q < \frac{N+2}{N-2}$. We prove some counterparts of the works of Gidas-Spruck and Aviles in the framework of boundary singularities.

When \eqref{E1} is replaced by an equation with an absortion term,
 \begin{equation}\label{E2}
-\Gd u+u^{q}=0 \quad \text{in } \Omega,
\end{equation}
the problem has been first adressed by Gmira-V\'eron~\cite{GmiVer:91} (and later to nonsmooth domains in \cite{FabVer:96}). These results are important in the theory of boundary trace of positive solutions of \eqref{E2} which was developed by Marcus-V\'eron~\cite{MarVer:98,MarVer:98a,MarVer:01} using analytic tools and by Le Gall~\cite{Leg:95} and Dynkin-Kuznetsov~\cite{DynKuz:96,DynKuz:98} with a probabilistic approach. We refer the reader to V\'eron~\cite{Ver:81} for the case of interior singularities of \eqref{E2}.

\medskip

Let us first consider the case where $\Omega$ is the upper-half space $\RR^N_+$, and we look
for solutions of \eqref{E1} of the form
\[
u(x) = {|x|}^{-\frac{2}{q-1}} \omega\big(\tfrac{x}{|x|}\big).
\]
By an easy computation, $\omega$ must satisfy
\begin{equation}\label{En3}
\left\{
\begin{alignedat}{2}
-\Delta'\omega & =\ell_{N,q}\omega + \omega^q && \quad \mbox{in }S^{N-1}_+,\\
\omega & \geq 0 && \quad\mbox{in } S^{N-1}_+,\\
\omega & = 0 && \quad\mbox{on }\partial S^{N-1}_+,
\end{alignedat}
\right.
\end {equation}  
where $\Delta'$ denotes the Laplace-Beltrami operator in the unit sphere $S^{N-1}$, 
\[
\ell_{N,q}= \tfrac{2(N-q(N-2))}{(q-1)^2} \quad \text{and} \quad
S^{N-1}_+ = S^{N-1} \cap \RR^N_+.
\]

Concerning equation \eqref{En3}, we prove

\begin{thm}\label{thm1.0} 
\quad
\begin{itemize}
\item[$(i)$] If $1<q\leq q_1$, then \eqref{En3} admits no
  positive solution.
\item[$(ii)$] If $q_1<q< q_3$, then \eqref{En3} admits a unique
positive solution.
\item[$(iii)$] If $q\geq q_3$, then \eqref{En3} admits no positive solution.
\end{itemize}
\end{thm}

In Section~\ref{sec8} we study uniqueness of solutions of \eqref{En3} with $\ell_{N,q}$ replaced by any $\ell \in \RR$. The proofs are inspired from some interesting ideas taken from Kwong~\cite{Kwo:91} and Kwong-Li~\cite{KwoLi:92}. The nonexistence of solutions of \eqref{En3} when $q \ge q_3$ is based on a Poho\v zaev identity for spherical domains; see Theorem~\ref{poth} below.

\medskip

We now consider the case where $\Omega \subset \RR^N$ is a smooth domain such that $0 \in \partial\Omega$. Without loss of generality, we may assume that $- \boldsymbol \e_N$ is the outward unit normal vector of $\partial\Omega$ at $0$. We prove the following classification of isolated singularities of solutions of \eqref{E1}:

\begin{thm}\label{thm1.2}
Assume that $q_{1} < q < q_{2}$. If $u$ satisfies \eqref{E1}, then either $u$ can be continuously extended at $0$ or for every $\eps > 0$ there exists $\delta > 0$ such that if $x \in \Omega\setminus\{0\}$, $\frac{x}{|x|} \in S_+^{N-1}$ and $|x|< \delta$,
\begin{equation}\label{E6}
\Big||x|^{\frac{2}{q-1}} u(x) - \omega \big( \tfrac{x}{|x|} \big) \Big| < \eps,
\end{equation}
where $\omega$ is the unique positive solution of \eqref{En3}.
\end{thm}

When $q_2 < q < q_3$, we have a similar conclusion provided $u$ satisfies the estimate
\[
u(x)\leq C|x|^{-\frac{2}{q-1}} \quad \forall x\in \Gw,
\]
for some constant $C > 0$; see Proposition~\ref{cv1} below. In the critical case $q=q_{1}$ there is a superposition of the linear and nonlinear effects since their characteristic exponents $\frac{2}{q-1}$ and $N-1$ coincide. The counterpart of Theorem~\ref{thm1.2} in this case is the following:

\begin{thm}\label{thm1.2a}
Assume that $q= q_1$. If $u$ satisfies \eqref{E1}, then either $u$ can be continuously extended at $0$ or for every $\eps > 0$ there exists $\delta > 0$ such that if $x \in \Omega\setminus\{0\}$ and $|x|< \delta$,
\begin{equation}\label{E6a}
\Big||x|^{N-1} \big(\log{\tfrac{1}{|x|}} \big)^{\frac{N-1}{2}} u(x) - \kappa \tfrac{x_N}{|x|}  \Big| < \eps,
\end{equation}
where $\kappa$ is a positive constant depending only on the dimension $N$.
\end{thm}

\medskip
Our characterization of boundary isolated singularities is complemented by the existence of singular solutions which has been recently obtained by del~Pino-Musso-Pacard \cite{DelMusPac:07}. We recall their result:

\begin{thm}\label{thm1.3}
Assume that $\Omega \subset \RR^N$ is a smooth bounded domain. There exists $p \in ( q_1,q_{2})$ such that for every $q_1 \leq q < p$ and for every $\xi_{1},\xi_2,\ldots,\xi_{k}\in\prt\Gw$, there exists a positive function $u\in C(\bar\Omega\setminus\{\xi_j\}_{j=1}^k)$, vanishing on $\prt\Gw\setminus\{\xi_j\}_{j=1}^k$, solution of $-\Delta u=u^q$ in $\Omega$,
 such that 
\[
u(x)\to + \infty \quad \text{as $x\to\xi_{j}$ nontangentially for every $i = 1,2,\ldots,k$.}
\]
\end{thm}

In view of Theorems~\ref{thm1.2} and \ref{thm1.2a} any such solution must have the singular behavior we have obtained therein. In \cite{DelMusPac:07}, the authors conjecture that such solutions exist for every $q_1 \leq q < q_2$.

\bigskip

Some of the main ingredients in the proofs of Theorems~\ref{thm1.2} and \ref{thm1.2a} are Theorem~\ref{thm1.0} above concerning existence and uniqueness of positive solutions of \eqref{En3}, a removable singularity result (see Theorems~\ref{prop6.1} and \ref{prop6.2} below) and the following a priori bound of solutions of \eqref{E1}:

\begin{thm}\label{thm1.1}
Assume that $1<q<q_{2}$. Then, every solution of \eqref{E1} satisfies
\begin{equation}\label{E5}
u(x)\leq C|x|^{-\frac{2}{q-1}} \quad \forall x\in \Gw,
\end{equation}
for some constant $C>0$ independent of the solution.
 \end{thm}
 
We establish this estimate using a topological argument, called the \emph{Doubling lemma} (see Lemma~\ref{doubl} below), introduced by Pol\'a\v cik-Quittner-Souplet~\cite{PolQuiSou:07}.

\medskip
Theorems~\ref{thm1.2}, \ref{thm1.2a} and \ref{thm1.1} have been announced in \cite{BidPonVer:07}.


\section{Poho\v zaev identity in spherical domains}\label{sec7}

We first prove the following Poho\v zaev identity in spherical domains.

\begin{thm}\label{poth}
Let $q > 1$, $\ell \in \RR$ and $S$ be a smooth domain in $S_+^{N-1}$. If $v \in C^2(S) \cap C(\overline{S})$ satisfies
\begin{equation}\label{Amain}
\left\{
\begin{alignedat}{2}
-\Delta'v & = \ell v+\abs v^{q-1}v  && \quad \text {in }S,\\
v & =0 && \quad\text {on }\partial S,
\end{alignedat}
\right.
\end{equation}
then
\begin{equation}\label{A}
\big(\tfrac{N-3}{2}-\tfrac{N-1}{q+1}\big)
\myint{S}{}\abs{\nabla' v}^2\gf\,d\gs
- \tfrac{N-1}{2}\big(\tfrac{\ell (q-1) + N-1}{q+1} \big) \myint{S}{} v^{2}\gf\,d\gs 
= \tfrac{1}{2}\myint{\prt S}{}\abs{\nabla' v}^2\langle\nabla'\gf,\gn\rangle\,d\gt,
\end{equation}
where $\nu$ is the outward unit normal vector on $\partial S$, $\nabla'$ the tangential gradient to $S^{N-1}$, and $\phi$ is a first eigenfunction of the Laplace-Beltrami operator $-\Delta'$ in $W_0^{1, 2}(S^{N-1}_+)$.
\end{thm}
 
We recall that the first eigenvalue of $-\Delta'$ in $W_0^{1, 2}(S^{N-1}_+)$ is $N-1$ and the eigenspace associated to this eigenvalue is spanned by the function $\phi(x) = \frac{x_N}{|x|}$.
 
\begin{proof}
Let
$$
P=\langle\nabla'\gf,\nabla' v\rangle\nabla' v.
$$
By the Divergence theorem,
\begin{equation}\label{7.1}
\int_{S} \div{P} \, d\gs= \int_{\prt S} \langle P,\gn\rangle\, d\gt.
\end{equation}
Note that
\begin{equation*}
\div{P} = \langle \nabla' v, \nabla'\gf \rangle\Gd' v+D^2v(\nabla' v, \nabla' \gf) + D^2\gf(\nabla' v,\nabla' v).
\end{equation*}
where $D^2v$ is the Hessian operator. Now,
\begin{equation*}
D^2v(\nabla' v, \nabla' \gf)=\myfrac{1}{2}\langle \nabla'\abs{\nabla' v}^2,\nabla' \gf\rangle.
\end{equation*}
Using the classical identity
\begin{equation*}
D^2\gf+\gf\,g=0
\end{equation*}
where $g=(g_{i,j})$ is the metric tensor on $S^{N-1}$, we get
\[
D^2\gf(\nabla' v,\nabla' v) = - g(\nabla' v, \nabla' v) \phi = - |\nabla' v|^2 \phi.
\]
We replace these identities in the expression of $\div{P}$,
\begin{equation*}
\div{P} = -\langle\nabla' v, \nabla'\gf\rangle \big(\ell v+\abs v^{q-1}v\big)
+\frac{1}{2}\langle\nabla' \abs{\nabla' v}^2,\nabla' \gf\rangle
-\abs{\nabla' v}^2 \gf.
\end{equation*}
Integrating over $S$, we obtain
$$
\int_{S} \div{P} \, d\gs = - \int_{S} \langle\nabla' v, \nabla'\gf \rangle \big(\ell v+\abs v^{q-1}v\big)
\,d\gs + \frac{1}{2}\int_{S} \langle\nabla' \abs{\nabla' v}^2,\nabla' \gf\rangle\,d\gs
-\int_{S} \abs{\nabla' v}^2 \gf\,d\gs.
$$
Note that
\begin{equation*}
\begin{aligned}
\int_{S} \langle\nabla' v, \nabla'\gf \rangle \big(\ell v+\abs v^{q-1}v \big)\,d\gs
& = \int_{S} \Big\langle \nabla' \Big(\tfrac{\ell}{2} v^2 + \tfrac{1}{q+1}\abs v^{q+1} \Big), \nabla'\gf \Big\rangle\,d\gs\\
& = -\int_{S} \Big(\tfrac{\ell}{2} v^2 + \tfrac{1}{q+1}\abs v^{q+1} \Big)\Gd'\gf\,d\gs\\
& = (N-1) \int_{S} \Big(\tfrac{\ell}{2} v^2 + \tfrac{1}{q+1}\abs v^{q+1} \Big) \gf \,d\gs,
\end{aligned}
\end{equation*}
and
\begin{equation*}
\begin{aligned}
\int_{S} \langle\nabla' \abs{\nabla' v}^2,\nabla' \gf\rangle\,d\gs
& = - \int_{S} \abs{\nabla' v}^2\Gd' \gf\,d\gs+
\int_{\prt S} \abs{\nabla' v}^2\langle\nabla'\gf,\gn\rangle\,d\gt\\
& = (N-1)\int_{S} \abs{\nabla' v}^2 \gf \,d\gs+
\myint{\prt S}{}\abs{\nabla' v}^2\langle\nabla'\gf,\gn\rangle\,d\gt.
\end{aligned}
\end{equation*}
These identities imply
\begin{multline}\label{7.2}
\int_{S} \div{P} \,d\gs =
-\tfrac{\ell (N-1)}{2} \int_{S} v^2 \phi \, d\sigma - \tfrac{N-1}{q+1} \int_{S} \abs v^{q+1} \gf\,d\gs+
\tfrac{N-3}{2} \int_{S} \abs{\nabla' v}^2 \gf\,d\gs+\\
+ \tfrac{1}{2} \int_{\prt S}\abs{\nabla' v}^2\langle\nabla'\gf,\gn\rangle\,d\gt.
\end{multline}
On the other hand, since $v$ satisfies \eqref{Amain},
\begin{equation*}
\begin{split}
\myint{S}{} \big(\ell v^2 + \abs v^{q+1} \big)\gf\,d\gs
& =-\myint{S}{} (\Gd' v) v\gf\,d\gs\\
& =\myint{S}{}\langle \nabla' v, \nabla' (v\gf) \rangle\,d\gs =\myint{S}{}\abs{\nabla' v}^2\gf\,d\gs+\myint{S}{}\langle \nabla' v, \nabla'\gf \rangle v\,d\gs.
\end{split}
\end{equation*}
Since $v\nabla' v = \frac{1}{2} \nabla'{(v^2)}$ and $\Gd'\gf=-(N-1)\gf$,
\[
\myint{S}{}\langle \nabla'v,\nabla' \gf\rangle v\,d\gs = \frac{1}{2} \int_S \langle \nabla'(v^2), \nabla'\phi \rangle\, d\sigma = \tfrac{N-1}{2} \int_S v^2 \phi \, d\sigma.
\]
Thus,
\begin{equation*}
\myint{S}{} \big(\ell v^2 + \abs v^{q+1} \big)\gf\,d\gs =\myint{S}{}\abs{\nabla' v}^2\gf\,d\gs+\tfrac{N-1}{2}\myint{S}{}v^2\gf\,d\gs.
\end{equation*}
This implies
$$
\myint{S}{}\abs v^{q+1}\gf\,d\gs=
\myint{S}{}\abs{\nabla' v}^2\gf\,d\gs+ \big(\tfrac{N-1}{2}-\ell \big) \myint{S}{}v^2\gf\,d\gs.
$$
Inserting this identity in \eqref{7.2}, we obtain
\begin{multline}\label{7.3}
\myint{S}{} \div{P} \,d\gs= \big(\tfrac{N-3}{2}-\tfrac{N-1}{q+1}\big)
\myint{S}{}\abs{\nabla' v}^2\gf\,d\gs
- \big(\tfrac{\ell (N-1)}{2}+\tfrac{N-1}{q+1}\big(\tfrac{N-1}{2}-\ell \big) \big) \myint{S}{} v^{2}\gf\,d\gs +\\
+\tfrac{1}{2}\myint{\prt S}{}\abs{\nabla' v}^2\langle\nabla'\gf,\gn\rangle\,d\gt.
\end{multline}
Since $v$ vanishes on $\prt S$, $\nabla' v= \langle \nabla' v,\gn\rangle \nu$ and, in particular, $|\nabla' v| = |\langle \nabla' v, \nu \rangle |$.
Thus,
\begin{equation}\label{7.4}
\begin{split}
\int_{\partial S} \langle P, \nu \rangle \, d\tau = \myint{\prt S}{}\langle\nabla'\gf,\nabla' v\rangle\langle\nabla' v,\gn\rangle\,d\gt
& = \myint{\prt S}{}\big(\langle\nabla' v,\gn\rangle \big)^2 \langle\nabla'\gf, \nu \rangle \,d\gt\\
& =\myint{\prt S}{}\abs{\nabla' v}^2\langle\nabla'\gf,\gn\rangle\,d\gt.
\end{split}
\end{equation}
Combining \eqref{7.1}, \eqref{7.3} and \eqref{7.4}, we get the Poho\v zaev identity.
\end{proof} 

Using the Poho\v zaev identity on $S^{N-1}_+$ we can prove that the Dirichlet problem \eqref{Amain} can only have trivial solutions for suitable values of $q$ and $\ell$.

\begin{cor}\label{cor7.1}
Let $N \ge 4$. If $q \geq q_3$ and $\ell \leq - \frac{N-1}{q-1}$, then the function identically zero is the only solution in $C^2(S^{N-1}) \cap C^2(\overline{S^{N-1}})$ of the Dirichlet problem
\begin{equation*}
\left\{
\begin{alignedat}{2}
-\Delta'v & = \ell v+\abs v^{q-1}v  && \quad \text {in }S^{N-1}_+,\\
v & =0 && \quad\text {on }\partial S^{N-1}_+.
\end{alignedat}
\right.
\end{equation*}
\end{cor}

\begin{proof}
Let $v$ be a solution of the Dirichlet problem. Applying the Poho\v zaev identity with $\phi(x) = \frac{x_N}{|x|}$, then the left-hand side of the Poho\v zaev identity is nonnegative, while its right-hand side is nonpositive. Thus, both sides are zero. If at least one of the inequalitites $q \geq q_3$ or $\ell \leq - \frac{N-1}{q-1}$ is strict, then we immediately deduce that $v = 0$ in $S_+^{N-1}$.\\
If $q = q_3$ and $\ell = - \frac{N-1}{q-1}$, then 
\[
\int\limits_{\prt S_+^{N-1}}\abs{\nabla' v}^2\langle\nabla'\gf,\gn\rangle\,d\gt = 0.
\]
Since $\langle \nabla'\gf,\gn\rangle < 0$ on $\partial S_+^{N-1}$, we conclude that $\nabla' v = 0$ on $\partial S_+^{N-1}$. Define the function $\tilde v : S^{N-1} \to \RR$ by 
\begin{equation*}
\tilde v(x) =
\begin{cases}
v(x) 	& \text{if } x \in S_+^{N-1},\\
0 	& \text{otherwise.}
\end{cases}
\end{equation*}
Then, $\tilde v$ satisfies (in the sense of distributions)
\begin{equation*}
-\Delta' \tilde v = \ell \tilde v+\abs{\tilde v}^{q-1} \tilde v  \quad \text {in }S^{N-1}.
\end{equation*}
Since $\tilde v$ vanishes in an open subset of $S^{N-1}$, by the unique continuation principle we have $\tilde v = 0$ in $S^{N-1}$ and the conclusion follows.
\end{proof}

\noindent\Remark When $S\subsetneq S_{+}^{N-1}$ and $q>q_{3}$ the previous non-existence result can be improved if we define
\begin{equation}\label{O}
\gl(S,\gf)=\sup\left\{\mu\geq 0:\myint{S}{}|\nabla'\gz|^2\gf d\gs\geq\mu
\myint{S}{}\gz^2\gf d\gs,\;\;\forall\gz\in C^\infty_{0}(S)\right\}.
\end{equation}
This constant $\gl(S,\gf)$ is actually zero if $S= S_{+}^{N-1}$. With this inequality (\ref{A}) turns into 
\begin{equation}\label{B}
\left[\big(\tfrac{N-3}{2}-\tfrac{N-1}{q+1}\big)\gl(S,\gf)
- \tfrac{N-1}{2}\big(\tfrac{\ell (q-1) + N-1}{q+1} \big) \right]\myint{S}{} v^{2}\gf\,d\gs 
\leq \tfrac{1}{2}\myint{\prt S}{}\abs{\nabla' v}^2\langle\nabla'\gf,\gn\rangle\,d\gt.
\end{equation}
Therefore, the statement of Corollary \ref{cor7.1}still holds if $q>q_{3}$ and
\begin{equation}\label{C}
\ell (q-1)\geq 1-N+\frac{q(N-3)-N-1}{N-1}\gl(S,\phi).
\end{equation}
Note that $\gl(S,\phi)$ tends to infinity if $S$ shrinks to a point.
 
\section{Uniqueness of solutions of a \textsc{pde} in $S_+^{N-1}$}\label{sec8}

In this section we address the question of uniqueness of positive solutions of the Dirichlet problem
\begin{equation}\label{8.1}
\left\{
\begin{alignedat}{2}
-\Gd'v & = \ell v + v^{q} && \quad\text{in } S^{N-1}_+,\\
v & \ge 0  && \quad\text{in } S^{N-1}_+,\\
v & = 0 && \quad\text{on } \prt S^{N-1}_+,
\end{alignedat}
\right.
\end{equation}
where $\ell \in \RR$. A solution of \eqref{8.1} is understood in the classical sense.

\medskip
We shall prove the following results:

\begin{thm}\label{thm8.1}
Assume that $N=2$. If $q >1$, then for every $\ell \in \RR$ the Dirichlet problem \eqref{8.1} has at most one positive solution.
\end{thm}

\begin{thm}\label{thm8.2}
Assume that $N \ge 4$. If $1 < q < q_3$, then for every $\ell \in \RR$ the Dirichlet problem \eqref{8.1} has at most one positive solution.
\end{thm}

\begin{thm}\label{thm8.3}
Assume that $N=3$. Then, the Dirichlet problem \eqref{8.1} has at most one positive solution under one of the following assumptions:
\begin{itemize}
\item for every $1 < q \le 5$ and $\ell \in \RR$,
\item for every $q > 5$ and $\ell \le \frac{2(3-q)}{(q+3)(q-1)}$.
\end{itemize}
\end{thm}

\begin{rem}
\rm
In dimension $N =3$ we do not know whether the Dirichlet problem \eqref{8.1} has a unique positive solution if $q > 5$ and $\ell > \frac{2(3-q)}{(q+3)(q-1)}$.
\end{rem}

We first show that the graphs of two positive solutions of \eqref{8.1} must cross.

\begin{lemma}\label{lemma8.0}
Assume that $v_1$ and $v_2$ are positive solutions of \eqref{8.1}. If $v_1 \le v_2$ in $S^{N-1}_+$, then $v_1 = v_2$.
\end{lemma}

\begin{proof}
Multiplying by $v_2$ the equation satisfied by $v_1$ and integrating by parts, we get
\[
\int\limits_{S^{N-1}_+} \langle\nabla v_1, \nabla v_2 \rangle \, d\sigma=  \int\limits_{S^{N-1}_+} \big(\ell v_1 + (v_1)^q \big) v_2  \, d\sigma.
\]
Reversing the roles of $v_1$ and $v_2$, we also have
\[
\int\limits_{S^{N-1}_+} \langle\nabla v_2, \nabla v_1 \rangle \, d\sigma=  \int\limits_{S^{N-1}_+} \big(\ell v_2 + (v_2)^q \big) v_1  \, d\sigma.
\]
Subtracting these identities, we have
\[
\int\limits_{S^{N-1}_+} \big({v_1}^{q-1} - {v_2}^{q-1}\big) v_1 v_2 \, d\sigma = 0.
\]
Since the integrand is nonnegative we must have ${v_1}^{q-1} - {v_2}^{q-1} = 0$ and the conclusion follows.
\end{proof}

We consider first the case $N = 2$. The precise structure of of the set of all signed solutions defined on $\BBR$ is already established in \cite[Lemma 1.1]{BidBo:94}, see also Theorem 1.1 therein for the main result. In this paper the proof is based upon the fact that the equation is autonomous. Here we use another argument which is in the line of the one developed in the cases $N\geq 3$ studied below.

\begin{proof}[Proof of Theorem~\ref{thm8.1}]
Denoting by
\[
\gth=\arccos{\tfrac{x_2}{|x|}}, 
\]
then a solution of \eqref{8.1} satisfies
\[
\left\{
\begin{aligned}
& v_{\theta\theta}+\ell v+v^q=0\quad\text {in } \big(0,\tfrac{\pi}{2} \big),\\
& v_\theta(0)=0,\quad v(\tfrac{\pi}{2})=0.
\end{aligned}
\right.
\]
Moreover, for every $\theta \in (0, \frac{\pi}{2}]$, $v_\theta(\theta) < 0$;  indeed $v_{\theta}(\frac{\pi}{2})<0$, and if $\rho=\inf\{\theta>0:v_\theta(\theta)<0 \}>0$, then from uniqueness $\rho=\frac{\pi}{4}$, $v(\theta)=v(\frac{\pi}{2}-\theta)$, hence $v_\theta(0)> 0$, contradiction (notice that this argument is the 1-dim moving plane method). Thus, $v$ is decreasing. Let $V : [0, v(0)] \to \RR$ be the function defined by 
\begin{equation}\label{8.1a}
V(\xi) = v_\theta(v^{-1}(\xi)).
\end{equation}
Then, $V$ is of class $C^1$ in $[0, v(0))$.
Since for every $\xi \in [0, v(0))$,
\[
(v^{-1})_\xi(\xi) = \frac{1}{v_\theta(v^{-1}(\xi))} = \frac{1}{V(\xi)},
\]
we deduce that
\begin{equation}\label{8.1b}
(V^2)_\xi = 2 V V_\xi = 2 V (v_{\theta\theta}\circ v^{-1}) (v^{-1})_\xi =  2 (v_{\theta\theta}\circ v^{-1}) = -2(\ell \xi + \xi^q).
\end{equation}
Assume by contradiction that \eqref{8.1} has two distinct positive solutions, say $v_1$ and $v_2$. We may assume they are both defined in terms of the variable $\theta$. Then, there exists $c_1 \in (0, \frac{\pi}{2})$ such that $v_1(c_1) = v_2(c_1)$. Let $c_2 \in (c_1, \frac{\pi}{2}]$ be the smallest number such that $v_1(c_2) = v_2(c_2)$ (this point $c_2$ exists since ${v_1}_\theta(c_1) \ne {v_2}_\theta(c_1)$). Without loss of generality, we may assume that, for every $\theta \in (c_1, c_2)$, 
\[
v_1(\xi) < v_2(\xi).
\] 
Let $V_1$ and $V_2$ be the functions given by \eqref{8.1a} corresponding to $v_1$ and $v_2$, respectively. For $i \in \{1, 2\}$, let
\[
\alpha_i = v_1(c_i) = v_2(c_i).
\]
By \eqref{8.1b}, for every $\xi \in (\alpha_2, \alpha_1)$,
\[
({V_1}^2)_\xi(\xi) = -2(\ell \xi + \xi^q) = ({V_2}^2)_\xi(\xi).
\]
Hence, the function ${V_1}^2 - {V_2}^2$ is constant. On the other hand, since $v_1 < v_2$ and $v_1, v_2$ are both decreasing, by uniqueness of the Cauchy problem,
\[
{v_1}_\theta(c_1) < {v_2}_\theta(c_1) < 0 \quad \text{and} \quad {v_2}_\theta(c_2) < {v_1}_\theta(c_2) < 0.
\] 
Thus,
\[
{V_1}^2(\alpha_1) - {V_2}^2(\alpha_1) > 0  \quad \text{and} \quad {V_1}^2(\alpha_2) - {V_2}^2(\alpha_2) < 0.
\]
This is a contradiction. We conclude that problem \eqref{8.1} cannot have more than one positive solution.
\end{proof}

\noindent\Remark The proofs in \cite{BidBo:94} as well as the one here are valid for equation
\begin{equation}\label{genE}
v_{\theta\theta}+\ell v+g(v)=0
\end{equation}
where $g\in C^1(\BBR)$, $g(0)=0$ and $r\mapsto\frac{g(r)}{r}$ is increasing on $(0,\infty)$. In \cite[Prop 4.4]{BidJaVer:08} a more general, result is obtained.
\medskip

In order to study \eqref{8.1} in the case of higher dimensions, the first step is to rewrite the Dirichlet problem in terms of an \textsc{ode}. By an adaptation of the moving planes method to $S^{N-1}$ (see \cite{Pad:97}), any positive solution $v$ of (\ref{8.1}) depends only on the geodesic distance to the North pole:
\[
\gth=\arccos{\tfrac{x_N}{|x|}}
\]
and $v$ decreasing with respect to $\theta$. Since in this case
$$
\Gd'v=\myfrac{1}{(\sin{\gth})^{N-2}} \frac{d}{d\theta}\left((\sin{\gth})^{N-2} v_{\gth}\right),
$$
every solution of \eqref{8.1} satisfies the following \textsc{ode} in terms of the variable $\theta$:
\begin{equation}\label{8.2}
\left\{
\begin{aligned}
& v_{\theta\theta}+(N-2)\cot{\gth}\, v_\theta+\ell v+v^q=0\quad\text {in } \big(0,\tfrac{\pi}{2} \big),\\
& v_\theta(0)=0,\quad v(\tfrac{\pi}{2})=0.
\end{aligned}
\right.
\end{equation}
The heart of the matter is then to apply some ideas from Kwong~\cite{Kwo:91} and Kwong-Li~\cite{KwoLi:92}, originally dealing with positive solutions of
\begin{equation}\label{Un0}
\left\{
\begin{aligned}
& u_{rr}+ (N-2)\frac{1}{r} \, u_r + \ell u + u^q=0 \quad\text {in }(0, a),\\
& u_r(0)=0, \quad u(a)=0.
\end{aligned}
\right.
\end{equation}

By Lemma~\ref{lemma8.0} and the discussion above, the graphs of two positive solutions of \eqref{8.2} must intersect in $(0, \frac{\pi}{2})$. Of course, the number of intersection points could be arbitrarily large (but always finite in view of the uniqueness of the Cauchy problem). The next lemma allows us to reduce the problem to the case where there could be only one intersection point. The argument relies on the shooting method and continuous dependence arguments; we only give a sketch of the proof.

\begin{lemma}\label{lemma8.0a}
Assume that \eqref{8.2} has two distinct positive solutions. Then, there exists two positive solutions  of \eqref{8.2} the graph of which  intersect only once in the interval $(0, \frac{\pi}{2})$.
\end{lemma}

\begin{proof}[Sketch of the proof]
For each $\alpha >0$ let $v^\alpha$ be the (unique) maximal solution of
\begin{equation*}
\left\{
\begin{aligned}
& v_{\theta\theta}+(N-2)\cot{\gth}\, v_\theta+\ell v+|v|^{q-1}v=0\quad\text {in } I_{\alpha}=(0,m_{\ga})\subset \big(0,\pi \big),\\
& v_\theta(0)=0,\quad v(0)= \alpha.
\end{aligned}
\right.
\end{equation*}
Then $v=v^\alpha$ is obtained by the contraction mapping principle on some interval $[0,\tau_{\alpha}]$, by the formula
\begin{equation}\label{fix}
v(\theta)=\alpha-\myint{0}{\theta}(\sin\sigma)^{2-N}\myint{0}{\sigma}(\sin\tau)^{N-2}(\ell v+|v|^{q-1}v)(\tau)d\tau d\sigma.
\end{equation}
It is extended to its maximal interval $I_{\alpha}$, and by a standard concavity argument, $m_{\ga}=\sup I_{\alpha}=\pi$. Notice that only a solution which vanishes at $\theta=\frac{\pi}{2}$ can be extended by continuity at $\theta=\pi$. By a standard argument $v^{\alpha}(\theta)$ depends continuously on $\alpha$, uniformly when $\theta\in [0,\pi-\epsilon_{0}]$ for any $\epsilon_{0}>0$ small enough, and $\epsilon_{0}=\frac{\pi}{4}$ will be good enough. Since $v_{\theta}=v^{\alpha}_{\theta}$ satisfies
\begin{equation*}
v_{\theta}(\theta)=-(\sin\theta)^{2-N}\myint{0}{\theta}(\sin\tau)^{N-2}(\ell v+|v|^{q-1}v)(\tau)d\tau,
\end{equation*}
it follows that $v^{\alpha}$ depends continuously of $\alpha$ in the $C^{1}([0,\frac{3\pi}{4}])$-topology. If $v_{1}$ and $v_{2}$ are two distinct solutions of (\ref{8.2}) we can suppose that $v_{2}(0)>v_{1}(0)$. We assume now that their graph have more than one intersection and denote by $\sigma_{1}$ and $\sigma_{2}$ respectively their first and second intersections in $(0,\frac{\pi}{2})$. If $\ga\in (0,v_{2}(0))$, we denote by 
$\gs_{j}(\ga)$, $j=1,2,...$, the finite and increasing sequence of intersections, if any, of the graphs of $v_{2}$ and $v^\ga$ in $(0,\frac{\pi}{2})$. Then $\gs_{1}(v_{1}(0))=\gs_{1}$ and $\gs_{2}(v_{1}(0))=\gs_{2}$. Since the derivatives of $v_{2}$ and $v^\ga$ at $\gs_{j}(\ga)$ differ, it follows from implicit function theorem that the mapping $\ga\mapsto \gs_{j}(\ga)$ is continuous. Then,  if $\sigma_{2} (\alpha)<\frac{\pi}{2}$ for any $\ga\in (0,v_{1}(0))$, $\sigma_{1} (\alpha)$ satisfies the same upper bound,. Since $v^\alpha\to 0$ uniformly on $[0,\frac{\pi}{2}]$ and $v_{2\,\theta}(\frac{\pi}{2})<0$, this implies 
$$\lim_{\ga\to 0}\sigma_{1} (\alpha)=\sigma_{2} (\alpha)=\frac{\pi}{2}.$$
By the mean value theorem there exists $\tau(\alpha)\in (\sigma_{2} (\alpha),\sigma_{1} (\alpha))$ where 
$v_{2\,\theta}(\tau(\alpha))=v^\alpha_{\theta}(\tau(\alpha))$. This is impossible as $\tau(\alpha)\to 0$ and 
$$\lim_{\ga\to 0}v_{2\,\theta}(\tau(\alpha))=v_{2,\theta}(\frac{\pi}{2})\neq \lim_{\ga\to 0}v^\alpha_{\theta}(\tau(\alpha))=0.$$\smallskip

Thus there exists $\tilde\ga\in (0,v_{1}(0))$ such that $\sigma_{2} (\tilde\ga)=\frac{\pi}{2}$. Moreover $\sigma_{1} (\tilde\ga)<\frac{\pi}{2}$ otherwhile we would have $v_{2\,\theta}(\frac{\pi}{2})=v_{\theta}^{\tilde\ga}\frac{\pi}{2}$ as above, and $v_{2}=v^{\tilde\ga}$.
Therefore $v^{\tilde\ga}$ is a solution of (\eqref{8.2}) which
 intersects only once $v_{2}$ in $(0,\frac{\pi}{2})$.
\end{proof}

The next result is standard but we present a proof for the convenience of the reader.

\begin{lemma}\label{lemma8.3}
Assume that $v_1$ and $v_2$ are positive solutions of \eqref{8.2} whose graphs coincide at a single point of $(0, \frac{\pi}{2})$. If $v_1(0) > v_2(0)$, then the function
\[
\theta \in (0, \tfrac{\pi}{2}) \longmapsto \frac{v_2(\theta)}{v_1(\theta)}
\]
is increasing.
\end{lemma}

\begin{proof}
Let $J : [0, \frac{\pi}{2}] \to \RR$ be the function defined as $J = v_1 {v_2}_\theta - v_2 {v_1}_\theta$. To prove the lemma, it suffices to show that $J > 0$ in $(0, \frac{\pi}{2})$. Using the equations satisfied by $v_1$ and $v_2$, one finds
\[
J_\theta = - (N-2) \cot{\theta} J + \big( {v_1}^{q-1} - {v_2}^{q-1} \big) v_1 v_2.
\]
Thus,
\[
\frac{1}{(\sin{\theta})^{N-2}} \big( (\sin{\theta})^{N-2} J \big)_\theta = \big( {v_1}^{q-1} - {v_2}^{q-1} \big) v_1 v_2.
\]
Let $\sigma \in (0, \frac{\pi}{2})$ be such that $v_1(\sigma) = v_2(\sigma)$. Since ${v_1}_\theta(\sigma) \ne {v_2}_\theta(\sigma)$, we have $v_1 > v_2$ in $(0, \sigma)$ and $v_1 < v_2$ in $(\sigma, \frac{\pi}{2})$, we conclude that the function
\[
\theta \in [0, \tfrac{\pi}{2}] \longmapsto (\sin{\theta})^{N-2} J(\theta)
\]
is increasing in $(0, \sigma)$ and decreasing in $(\sigma, \frac{\pi}{2})$. Since it vanishes at $0$ and $\frac{\pi}{2}$, we have 
\[
(\sin{\theta})^{N-2} J > 0 \quad \text{in $(0, \tfrac{\pi}{2})$.}
\]
Thus $J > 0$ in $(0, \tfrac{\pi}{2})$ and the conclusion follows.
\end{proof}

The following identity will be needed in the proofs of Theorems~\ref{thm8.2} and \ref{thm8.3}.
\begin{lemma}\label{lemma8.1}
Let $v$ be a solution of \eqref{8.2}, $\ga = \frac{2(N-2)}{q+3}$ and $\gb= \frac{2(N-2)(q-1)}{q+3}$.
Set
\begin{equation}\label{8.3a}
w(\gth)=(\sin{\theta})^\alpha \,v(\gth)
\end{equation}
Let $E:(0,\frac{\pi}{2})\mapsto \BBR$ and $G : (0, \frac{\pi}{2}) \to \RR$
be the functions defined by
\begin{equation}\label{8.3a'}
E(\theta)=(\sin\theta)^{\beta}\frac{w_{\gth}^2}{2}
+G(\gth)\frac{w^2}{2}+\frac{w^{q+1}}{q+1},
\end{equation}
\begin{equation}\label{8.2a}
G(\gth)=\Big(\big(\ga(N-2-\ga)+\ell \big)(\sin{\gth})^2 +\ga(\ga+3-N)\Big)(\sin\gth)^{\gb-2}.
\end{equation}
Then,
\begin{equation}\label{X}
E_{\gth}=G_{\gth}\frac{w^2}{2}.
\end{equation}
\end{lemma}
\begin{proof}
Let $w : (0, \frac{\pi}{2}) \to \RR$ be the function defined by \eqref{8.3a}.
Then,
\begin{equation*}
w_{\theta\theta} + (N-2-2\ga)\cot{\theta} \,w_\theta + \left(\ga(N-2-\ga)+\ell +\frac{\ga(\ga+3-N)}{(\sin{\theta})^2}\right)w + \frac{w^q}{(\sin{\theta})^{\ga(q-1)}} =0.
\end{equation*}
Multiplying this identity by $(\sin{\gth})^\gb$, we get
$$
(\sin{\gth})^\gb\, w_{\theta\theta} + (N-2-2\ga)(\sin{\theta})^{\gb-1} \cos{\theta} \, w_\theta
+ G(\theta) \, w + (\sin{\gth})^{\gb-\ga(q-1)}w^q =0
$$
where $G$  is defined by \eqref{8.2a}. We now observe that $\alpha$ and $\beta$ satisfy
\[
N-2-2\ga = \frac{\beta}{2} \quad \text{and} \quad \gb-\ga(q-1) = 0.
\]
The identity satisfied by $w$ becomes
\begin{equation*}
(\sin{\gth})^\gb\, w_{\theta\theta} + \frac{\beta}{2} (\sin{\theta})^{\gb-1} \cos{\theta} \, w_\theta
+ G(\theta) \, w + w^q =0.
\end{equation*}
Since
$$
\frac{d}{d\gth}\left((\sin{\gth})^\gb\myfrac{(w_\theta)^2}{2}\right)=\left((\sin{\gth})^\gb\, w_{\theta\theta}+ \frac{\gb}{2}(\sin{\gth})^{\gb-1} \cos{\gth} \, w_\theta \right)w_\theta
$$
and
$$
\myfrac{d}{d\gth}\left(G(\gth)\myfrac{w^2}{2}\right)= G(\gth)ww_\theta + G_\theta(\gth)\frac{w^2}{2}
$$
identity (\ref{X}) follows.
\end{proof}

The following proof is inspired from Kwong-Li~\cite{KwoLi:92}.

\begin{proof}[Proof of Theorem~\ref{thm8.2}]
We use the notation of Lemma~\ref{lemma8.1}. We observe that $E$ can be continuously extended at $0$ and $\frac{\pi}{2}$. This is clear at $\frac{\pi}{2}$, where we take 
\begin{equation}\label{Y}
E(\tfrac{\pi}{2}) = \frac{(w_\theta(\tfrac{\pi}{2}))^2}{2} = \frac{(v_\theta(\tfrac{\pi}{2}))^2}{2}.
\end{equation}
To reach the conclusion at $0$, it suffices to observe that for every $\theta \in (0, \frac{\pi}{2})$,
\[
\begin{split}
(\sin{\theta})^\beta(w_\theta(\theta))^2 
& = (\sin{\theta})^\beta\Big( \alpha (\sin{\theta})^{\alpha - 1} \cos{\theta} \, v(\theta) + (\sin{\theta})^\alpha\, v_\theta(\theta) \Big)^2\\
& = (\sin{\theta})^{2\alpha + \beta -2} \Big( \alpha \cos{\theta} \, v(\theta) + \sin{\theta} \, v_\theta(\theta) \Big)^2.
\end{split}
\]
Since $N \ge 4$,
\[
2\alpha + \beta - 2 = \tfrac{2(N-3)}{q+3} \big( q + \tfrac{N-5}{N-3} \big) > 0,
\]
the right-hand side of the previous expression converges to $0$ as $\theta \to 0$. We can then set
$E(0) = 0$. 
Notice that
\[
G_\theta(\gth)
= \Big[\big(\ga(N-2-\ga)+\ell\big)\gb (\sin{\gth})^2 + \ga(\ga+3-N)(\gb-2)\Big](\sin{\gth})^{\gb-3}\cos{\gth}.
\]
By the choices of $\alpha$ and $\beta$,
\[
\ga(\ga+3-N)(\gb-2) = \tfrac{4(N-2)(N-3)^2}{(q+3)^3}\big( q + \tfrac{N-5}{N-3} \big) \big( \tfrac{N+1}{N-3} - q  \big).
\]
Since $N \ge 4$ and $1 < q < \frac{N+1}{N-3}$, this quantity is positive. Hence, there exists $\eps > 0$ such that
\[
G_\theta (\theta)>0 \quad \forall \theta \in (0, \eps).
\]
In view of the expression of $G_\theta$, we have the following possibilities: either
\begin{itemize}
\item[$(i)$] $G_\theta > 0$ in $(0, \frac{\pi}{2})$,
\end{itemize}
or
\begin{itemize}
\item[$(ii)$] there exists $c\in (0, \frac{\pi}{2})$ such that $G_\theta > 0$ in $(0,c)$ and $G_\theta < 0$ in  $(c, \frac{\pi}{2})$.
\end{itemize}
Assume by contradiction that \eqref{8.1} has more than one solution, hence by Lemma~\ref{lemma8.0a} problem \eqref{8.2} has two positive solutions $v_1$ and $v_2$ whose graphs intersect exactly once in the interval $(0, \frac{\pi}{2})$. Without loss of generality, we may assume that $v_1(0) > v_2(0)$. For $i \in \{1, 2\}$, define $w_i$ and $E_i$ accordingly.

\medskip
First, assume that $G$ satisfies property $(i)$ above. 
Let
\[
\gamma = \frac{{v_2}_\theta(\frac{\pi}{2})}{{v_1}_\theta(\frac{\pi}{2})}.
\]
We have from (\ref{Y})
\begin{equation}\label{Z}
(E_2 - \gamma^2 E_1)(0) = 0=(E_2 - \gamma^2 E_1)(\frac{\pi}{2}).
\end{equation}
On the other hand, by Lemma~\ref{lemma8.3} the function
\[
\theta \in (0, \tfrac{\pi}{2}) \longmapsto \frac{v_2(\theta)}{v_1(\theta)}
\]
is increasing. In particular, for every $\theta \in [0, \frac{\pi}{2})$,
\[
\frac{v_2(\theta)}{v_1(\theta)} < \lim_{\theta \to \frac{\pi}{2}-}{\frac{v_2(\theta)}{v_1(\theta)}} = \frac{{v_2}_\theta(\frac{\pi}{2})}{{v_1}_\theta(\frac{\pi}{2})} = \gamma.
\]
Hence,
\[
(w_2)^2 - \gamma^2 (w_1)^2 = (\sin{\theta})^{2\alpha} \big( (v_2)^2 - \gamma^2 (v_1)^2 \big) < 0 \quad \text{in $\big(0, \tfrac{\pi}{2}\big)$}.
\]
Thus, by Lemma~\ref{lemma8.1} and by assumption $(i)$, we have for every $\theta \in (0, \frac{\pi}{2})$,
\[
(E_2 - \gamma^2 E_1)_\theta(\theta) = G_\theta(\theta) \big( (w_2)^2 - \gamma^2 (w_1)^2 \big) < 0.
\]
This contradicts (\ref{Z}). Therefore, problem \eqref{8.1} cannot have two distinct positive solutions if $G$ satisfies $(i)$.

\medskip
Next, we assume that $G$ satisfies property $(ii)$ for some point $c$.
Let
\[
\tilde\gamma = \frac{v_2(c)}{v_1(c)}.
\]
As in the previous case,
$(E_2 - \tilde\gamma^2 E_1)(0) = 0$.
By Lemma~\ref{lemma8.3}, we have
\[
\frac{v_2}{v_1} < \tilde\gamma \quad \text{in $(0, c)$} \quad \text{and} \quad \frac{v_2}{v_1} > \tilde\gamma \quad \text{in $(c, \tfrac{\pi}{2})$}.
\]
Hence $ (E_2 - \tilde\gamma^2 E_1)(\tfrac{\pi}{2}) > 0$.
By Lemma~\ref{lemma8.1} and by assumption $(ii)$, we have for every $\theta \in (0, \frac{\pi}{2})$,
\[
(E_2 - \tilde\gamma^2 E_1)_\theta(\theta) = G_\theta(\theta) \big( (w_2)^2 - \tilde\gamma^2 (w_1)^2 \big) \leq 0.
\]
This is still a contradiction. Therefore, if $G$ satisfies $(ii)$, then problem \eqref{8.1} has a unique positive solution. The proof of Theorem~\ref{thm8.2} is complete.
\end{proof} 

When $N=3$, the proof of uniqueness of positive solutions of \eqref{8.1} is inspired from Kwong-Li~\cite {KwoLi:92} (Case~1 below) and Kwong~\cite{Kwo:91} (Case~2 below).

\begin{proof}[Proof of Theorem~\ref{thm8.3}]
We split the proof in two cases:

\smallskip
\noindent
\textit{Case~1.} $q > 1$ and  $\ell \le \frac{2(3-q)}{(q+3)(q-1)}$.
\smallskip

Let $G : (0, \frac{\pi}{2}) \to \RR$ be the function defined by \eqref{8.2a}. Since $N = 3$, we have
$\alpha = \frac{2}{q+3}$ and $\beta = \frac{2(q-1)}{q+3}$. Thus,
\[
\ga(\ga+3-N)(\gb-2)=\ga^2(\gb-2)= - \tfrac{32}{(q+3)^3} <0.
\]
Moreover, since by assumption  $\ell \le \frac{2(3-q)}{(q+3)(q-1)}$, we have
\[
\big(\ga(N-2-\ga)+\ell\big)\gb + \ga(\ga+3-N)(\gb-2) = \tfrac{2(q-1)}{q+3} \Big[ \tfrac{2(q-3)}{(q+3)(q-1)} + \ell \Big] \le 0.
\]
Therefore, $G$ satisfies
\begin{itemize}
\item[$(iii)$] $G_\theta < 0$ in $(0, \frac{\pi}{2})$.
\end{itemize}
We still consider the function $E$ defined by (\ref{8.3a'}), and astisfying (\ref{X}).
We observe that $E$ can still be continuously extended at $\frac{\pi}{2}$ by (\ref{Y}), 
but not at $0$ since $E(\theta)$ diverges to $+\infty$ as $\theta \to 0$.

\smallskip
Assume by contradiction that \eqref{8.1} has more than one solution, hence as above problem \eqref{8.2} has two positive solutions $v_1$ and $v_2$ whose graphs intersect exactly once in the interval $(0, \frac{\pi}{2})$, and $v_1(0) > v_2(0)$. For $i \in \{1, 2\}$, define $w_i$ and $E_i$ accordingly.

Let
\[
\hat\gamma = \frac{v_2(0)}{v_1(0)}.
\]
By Lemma~\ref{lemma8.3} we find
\[
(w_2)^2 - \hat\gamma^2 (w_1)^2 = (\sin{\theta})^{2\alpha} \big( (v_2)^2 - \hat\gamma^2 (v_1)^2 \big) > 0 \quad \text{in } (0, \tfrac{\pi}{2}).
\]
By Lemma~\ref{lemma8.1} and by assumption $(iii)$, we have for every $\theta \in (0, \frac{\pi}{2})$,
\begin{equation}\label{8.3xbis}
(E_2 - \hat\gamma^2 E_1)_\theta(\theta) = G_\theta(\theta) \big( (w_2)^2 - \hat\gamma^2 (w_1)^2 \big) < 0.
\end{equation}
By Lemma~\ref{lemma8.3},
\[
(E_2 - \hat\gamma^2 E_1)(\tfrac{\pi}{2}) =  \frac{({v_2}_\theta(\tfrac{\pi}{2}))^2 - \hat\gamma^2 ({v_1}_\theta(\tfrac{\pi}{2}))^2}{2} > 0.
\]
Although $E_1$ and $E_2$ cannot be continuously extended at $0$, one checks that
\[
\lim_{\theta \to 0}{\big(E_2(\theta) - \hat\gamma^2 E_1(\theta) \big)} = 0,
\]
by expanding the $v_{i}$ up to the order $2$ at $\theta=0$. 
This contradicts \eqref{8.3xbis}. Therefore, equation \eqref{8.1} has at most one positive solution.

\smallskip
\noindent
\textit{Case~2.} $1 < q \le 5$ and  $\ell > \frac{2(3-q)}{(q+3)(q-1)}$.
\smallskip

Since $1 < q \le 5$, we have $\ell>-\frac{1}{8}$, in particular $\ell \ge - \frac{1}{4}$. The remaining of the argument only requires $1 < q \le 5$ and $\ell \ge - \frac{1}{4}$.\\
Let $z : (0, \frac{\pi}{2}) \to \RR$ be the function defined as
\[
z(\gth)=(\sin\gth)^\frac{1}{2} \,v(\gth).
\] 
Then, $z$ satisfies
\begin{equation}\label{8.5a}
z_{\theta\theta} + \bigg(\ell+\frac{1}{4}+ \frac{1}{4(\sin{\gth})^2}\bigg) z + \frac{z^q}{(\sin\gth)^{\frac{q-1}{2}}} = 0.
\end{equation}

Assume by contradiction that equation  \eqref{8.2} has two positive distinct solutions $v_{1}$ and $v_{2}$ intersecting at some point $\sigma_{0}\in (0,\frac{\pi}{2})$, with $v_{1}(0)>v_{2}(0)$. Define $z_{1}$ and $z_{2}$ accordingly. Then $z_{1}>z_{2}$ on $(0,\sigma_{0})$, $z_{1}<z_{2}$ on $(\sigma_{0},\frac{\pi}{2})$ and $z_{1}(0)=z_{2}(0)=z_{1}(\frac{\pi}{2})=z_{2}(\frac{\pi}{2})=0$. let $\xi_{0}=z_{1}(\sigma_{0})=z_{2}(\sigma_{0})$.
\smallskip

As a first claim, we show that $z_1$ and $z_2$ cannot be both decreasing in $[\sigma_0, \frac{\pi}{2}]$. Indeed, if it holds,  we may consider their inverses $z_i^{-1} : [0, \xi_0] \to [\sigma_0,\frac{\pi}{2}]$. For $i \in \{1, 2\}$, let $Z_i : [0, \xi_0] \to \RR$ be the function given by
\[
Z_i(\xi) = {z_i}_\theta(z_i^{-1}(\xi))
\]
($Z_i$ is well-defined since $\sigma_1 > 0$). Since
\[
{z_1}_\theta(\sigma_0) < {z_2}_\theta(\sigma_0) < 0 \quad \text{and} \quad {z_2}_\theta(\frac{\pi}{2}) < {z_1}_\theta(\frac{\pi}{2}) < 0,
\]
we have
\[
({Z_1}(\xi_0))^2 > ({Z_2}(\xi_0))^2 \quad \text{and} \quad ({Z_1}(0))^2 < ({Z_2}(0))^2 .
\]
From the Mean value theorem, there exists $\eta \in (0, \xi_0)$ such that
\begin{equation}\label{8.6}
({Z_1}^2)_\xi(\eta) > ({Z_2}^2)_\xi(\eta).
\end{equation}
On the other hand, for $i \in \{1, 2\}$ and for every $\xi \in (0, \xi_0)$,
\begin{equation}\label{H}
Z_i {Z_i}_\xi = {z_i}_{\theta\theta} (z_i^{-1}(\xi)) = - \bigg(\ell+\frac{1}{4}+ \frac{1}{4(\sin{z_i^{-1}(\xi)})^2}\bigg) \xi - \frac{\xi^q}{(\sin{z_i^{-1}(\xi))}^{\frac{q-1}{2}}}.
\end{equation}
Since $z_1^{-1}(\xi) < z_2^{-1}(\xi)$ in $(0, \xi_0)$, we deduce that
\[
({Z_1}^2)_\xi = 2 Z_1 {Z_1}_\xi < 2 Z_2 {Z_2}_\xi = ({Z_2}^2)_\xi.
\]
This contradicts \eqref{8.6} and prove the claim.
\smallskip

As a second claim, we now show that $z_1$ and $z_2$ cannot be both increasing in $(0, \sigma_0)$. Assuming that it holds, we may consider their inverses $z_i^{-1} : [0, \xi_0] \to [0, \sigma_0]$. For $i \in \{1, 2\}$, let $Y_i : [0, \xi_0] \to \RR$ be the function defined as
\[
Y_i(\xi)  = {z_i}_\theta(z_i^{-1}(\xi)) (\sin{z_i^{-1}(\xi)}).
\]
Observe that $Y_i$ can be continuously extended to $0$ by taking $Y_i(0) = 0$. 
Since $z_2 < z_1$ in $(0, \sigma_0)$, we have
\begin{equation}\label{8.6a}
({Y_1}(0))^2 = ({Y_2}(0))^2=0 \quad \text{and} \quad ({Y_2}(\xi_0))^2 > ({Y_1}(\xi_0))^2.
\end{equation}
%
On the other hand, for $i \in \{1, 2\}$,
\[
\begin{split}
{Y_i}_\xi 
& = \Big( {z_i}_{\theta\theta} (z_i^{-1}(\xi)) (\sin{z_i^{-1}(\xi)}) + {z_i}_\theta(z_i^{-1}(\xi)) (\cos{z_i^{-1}(\xi)}) \Big) \frac{1}{{z_i}_\theta(z_i^{-1}(\xi))}\\
& = \Big( {z_i}_{\theta\theta} (z_i^{-1}(\xi)) (\sin{z_i^{-1}(\xi)})^2 \Big) \frac{1}{Y_i}  + \cos{z_i^{-1}(\xi)}.
\end{split}
\]
Thus,
\[
Y_i {Y_i}_\xi - Y_i \cos{z_i^{-1}(\xi)} = - \Big((\ell+\tfrac{1}{4})(\sin{z_i^{-1}(\xi)})^2 + \tfrac{1}{4} \Big) \xi - (\sin{z_i^{-1}(\xi))}^{\frac{5-q}{2}} \xi^q.
\]
Since $z_2^{-1}(\xi) > z_1^{-1}(\xi) $ in $(0, \xi_0)$ and $\ell \ge -\frac{1}{4}$, 
\[
(\ell+\tfrac{1}{4})(\sin{z_2^{-1}(\xi)})^2 \ge (\ell+\tfrac{1}{4})(\sin{z_1^{-1}(\xi)})^2.
\]
Since $q \le 5$,
\[
(\sin{z_2^{-1}(\xi))}^{\frac{5-q}{2}} \ge (\sin{z_1^{-1}(\xi))}^{\frac{5-q}{2}}.
\]
We deduce that
\[
Y_2 {Y_2}_\xi - Y_2 \cos{z_2^{-1}(\xi)} \le Y_1 {Y_1}_\xi - Y_1 \cos{z_1^{-1}(\xi)}.
\]
Hence,
\[
\begin{split}
\big( (Y_2)^2 - (Y_2)^2  \big)_\xi
& \le 2 (Y_2 \cos{z_2^{-1}(\xi)} - Y_1 \cos{z_1^{-1}(\xi)})\\
& \le 2\cos{z_1^{-1}(\xi)} (Y_1 - Y_2)\\
& \le \frac{2\cos{z_1^{-1}(\xi)}}{Y_1 + Y_2} \big((Y_1)^2 - (Y_2)^2 \big).
\end{split}
\]
Let $f: (0, \xi_0) \to \RR$ be the function defined by
\[
f(\xi) = \frac{2\cos{z_1^{-1}(\xi)}}{Y_1(\xi) + Y_2(\xi)}.
\]
Using this notation,
\[
\big( (Y_2)^2 - (Y_1)^2  \big)_\xi \le f(\xi) \, \big((Y_2)^2 - (Y_1)^2 \big).
\]
Thus, for every $\xi \in [0, \xi_0]$,
\[
\big( (Y_2)^2 - (Y_1)^2  \big)(\xi) \ge  \big( (Y_2)^2 - (Y_1)^2  \big)(\xi_0) \, \e^{\int_\xi^{\xi_0} f(\tau) \, d\tau}.
\]
This clearly contradicts \eqref{8.6a} and the second claim is proved.

\medskip
We can now conclude the proof. It follows from equation \eqref{8.5a} that both $z_1$ and $z_2$ are concave. Since $z_{1}$ and $z_{2}$ cannot be simultaneously increasing on $(0,\sigma_{0})$ or decreasing on $(\sigma_{0},\frac{\pi}{2})$, at their intersection point there holds
\[
{z_1}_\theta(\sigma_{0}) <0< {z_2}_\theta(\sigma_{0}).
\]
Therefore, the maximum of $z_1$ is achieved in $(0, \sigma_{0})$ while the maximum of $z_2$ is achieved in $(\sigma_{0}, \frac{\pi}{2})$.

\medskip
Denote the maximum of $z_i$ by $m_i$. We first show that $m_2 > m_1$. Indeed, assume by contradiction that $m_2 \le m_1$. Let $\tilde\sigma_2 \in (\sigma_{0}, \frac{\pi}{2})$ be such that 
\[
z_2(\tilde\sigma_2) = m_2.
\]
Let $\tilde\sigma_1$ be the largest number in $(0, \frac{\pi}{2})$ such that 
\[
z_1(\tilde\sigma_1) = m_2.
\]
The restrictions $z_i : [\tilde\sigma_i, \frac{\pi}{2}] \to [0, m_2]$ are both decreasing.
Let $\tilde Z_i : [0, m_2] \to \RR$ be the function defined as
\[
\tilde Z_i(\xi) = {z_i}_\theta(z_i^{-1}(\xi)).
\]
In the interval $[0,m_{2}]$ we have $z_1^{-1}(\xi) < z_2^{-1}(\xi)$ and (\ref{H}), thus, as in the first claim,
\[
({\tilde Z_1}^2)_\xi < ({\tilde Z_2}^2)_\xi.
\]
Since
\[
({\tilde Z_1}(0))^2 < ({\tilde Z_2}(0))^2 \quad \text{and} \quad ({\tilde Z_1}(m_2))^2 \geq 0 = ({\tilde Z_2}(m_2))^2 ,
\]
we have a contradiction.

\medskip
We now show that $m_1\geq m_2$. Assume by contradiction that $m_1 < m_2$. Let $\hat\sigma_1 \in (0, \sigma)$ be such that 
\[
z_1(\hat\sigma_1) = m_1.
\]
Let $\hat\sigma_2$ be the smallest number in $(0, \frac{\pi}{2})$ such that 
\[
z_2(\hat\sigma_1) = m_1.
\]
The restrictions $z_i : [0, \hat\sigma_i] \to [0, m_1]$ are both increasing.
Let $\hat Y_i : [0, m_1] \to \RR$ be the function defined as
\[
\hat Y_i(\xi) = {z_i}_\theta(z_i^{-1}(\xi)) (\sin{z_i^{-1}(\xi)})
\]
if $\xi \ne 0$ and $\hat Y_i(0) = 0$. Then, $\hat Y_i$ is continuous. In the interval $[0, \hat\sigma_i]$ we have $z_1^{-1}(\xi) < z_2^{-1}(\xi)$, thus, as in the second claim, 
\[
\big( (\hat Y_2)^2 - (\hat Y_1)^2  \big)_\xi \le \frac{2\cos{z_1^{-1}(\xi)}}{\hat Y_2 + \hat Y_1} \big((\hat Y_2)^2 - (\hat Y_1)^2 \big).
\]
This contradicts
\[
\big( (\hat Y_2)^2 - (\hat Y_1)^2 \big)(0) = 0 \quad \text{and} \quad \big( (\hat Y_2)^2 - (\hat Y_1)^2 \big)(m_1) > 0.
\]

\smallskip
Finally $m_2 > m_1\geq m_{1}>0$, which is a contradiction. Therefore, problem \eqref{8.1} can have at most one positive solution.
\end{proof}


\section{Proof of Theorem~\ref{thm1.0}}\label{sec9}

\begin{proof}[Proof of $(i)$]
Assume that $1 < q \le q_1$. Let $\phi$ be a positive eigenfunction of $-\Delta'$ in $W^{1, 2}_0(S^{N-1}_+)$ associated to the first eigenvalue $N-1$, and let $\omega$ be a solution of \eqref{En3}. Using $\phi$ as test function, we get 
\[
\int\limits_{S_+^{N-1}} \langle \nabla'\omega, \nabla'\phi \rangle \, d\sigma =  \int\limits_{S_+^{N-1}} (\ell_{N,q} \omega + \omega^q) \phi \, d\sigma.
\]
On the other hand, since $\phi$ is an eigenfunction of $-\Delta'$,
\[
\int\limits_{S_+^{N-1}} \langle \nabla'\omega, \nabla'\phi \rangle \, d\sigma =  (N-1) \int\limits_{S_+^{N-1}} \omega \phi \, d\sigma.
\]
Thus,
\begin{equation}\label{9.1}
(N-1 -\ell_{N, q})\int\limits_{S^{N-1}_+} \omega\gf \, d\gs =\int\limits_{S^{N-1}_{+}} \omega^q \gf\,d\gs.
\end{equation}
Since $q \le q_1$, we have
\[
N-1 - \ell_{N, q} = \tfrac{(N-1)(q+1)}{(q-1)^2} \big( q - \tfrac{N+1}{N-1} \big) \le 0.
\]
Hence, the left-hand side of \eqref{9.1} is nonpositive while the right-hand side is nonnegative. Thus,
\[
\int\limits_{S^{N-1}_{+}} \omega^q \gf\,d\gs = 0
\]
We conclude that $\omega = 0$ in $S^{N-1}_+$. Hence, problem \eqref{En3} has no positive solution.
\end{proof}

\begin{proof}[Proof of $(ii)$]
Since $q > q_1$,
\[
N-1 - \ell_{N, q} = \tfrac{(N-1)(q+1)}{(q-1)^2} \big( q - \tfrac{N+1}{N-1} \big) > 0.
\]
Thus, the functional $\CJ : W_0^{1,2}(S^{N-1}_+) \to \RR$ defined by
$$
\CJ(w)=\int\limits_{S^{N-1}_{+}} \big(\abs{\nabla' w}^2-\ell_{N,q} w^2 \big) \, d\gs
$$
is bounded from below by $0$. On the other hand, since $q < q_3$ we can minimize $\CJ$ over the set
$$
\bigg\{w\in W^{1,2}_{0}(S^{N-1}_{+}) \ ; \ \int\limits_{S^{N-1}_{+}} (w^+)^{q+1} \, d\gs = 1\bigg\}.
$$
Let $w$ be a minimizer. Then, $w^+$ is also a minimizer, whence $w= w^+$ and this function satisfies
\[
- \Delta'w - \ell_{N, q} w = \lambda w^q \quad \text{in $S^{N-1}_+$}
\]
for some $\lambda > 0$. By standard elliptic regularity theory, $w$ is smooth and vanishes on $\partial S^{N-1}_+$ in the classical sense. The function $\lambda^\frac{1}{q-1} w$ is therefore a solution of \eqref{En3}. For uniqueness, one applies Theorem \ref{thm8.1} and Theorem \ref{thm8.2} in the case $N\neq 3$. If $N=3$ we can applies Theorem \ref{thm8.3} since $\ell_{q,3}=\frac{2(3-q)}{(q-1)^2}$ always satisfies the assumption therein.
\end{proof}

\begin{proof}[Proof of $(iii)$]
We may assume that $N \ge 4$, for otherwise there is nothing to prove. Note that if $q \ge q_3$,
\[
\tfrac{N-1}{q-1} - \ell_{N,q} = - \tfrac{N-3}{(q-1)^2} \big( q - \tfrac{N+1}{N-3} \big) \le 0.
\]
Applying Corollary~\ref{cor7.1}, we deduce that \eqref{En3} has no positive solution.
\end{proof}


\section{The a priori estimate}\label{sec3}

In this section we establish Theorem~\ref{thm1.1} whose proof is based on the following result.

\begin{prop}\label{apr2}
Assume that $1<q<q_{2}$. Let $0 < r < \frac{1}{2}\diam{\Omega}$ and $\zeta \in C^\infty(\partial\Omega)$ with $\zeta \ge 0$ on $\partial\Omega$. Then, every solution of 
\begin{equation}\label{apr3}
\left\{
\begin{alignedat}{2}
-\Gd u 	&= u^q 	&& \quad\text{in }\Gw\cap (B_{2r}\setminus \overline B_{r}),\\
u 	& \ge 0 	&& \quad\text{in }\Gw\cap (B_{2r}\setminus \overline B_{r}),\\
u 	& = \gz 	&& \quad\text{on }\prt\Gw\cap (B_{2r}\setminus \overline B_{r}),
\end{alignedat}
\right.
\end{equation}
satisfies
\begin{equation}\label{apr4}
u(x)\leq C \big[ \dist(x, \Gamma_r)  \big]^{- \frac{2}{q-1}}
\quad\forall x \in \Gw\cap (B_{2r}\setminus \overline B_{r}),
\end{equation}
where $\Gamma_r = \overline\Omega \cap (\partial B_{2r} \cup \partial B_{r})$ and $C>0$ is a constant independent of $u$.
\end{prop}

We denote by $B_r$ the ball of radius $r$ centered at $0$. The proof of this estimate is based on two results: a Liouville theorem for the equation $-\Delta u = u^q$ in $\RR^N$ or in $\RR^N_+$ (see \cite{Dan:92}) and the Doubling lemma of Pol\'a\v cik-Quittner-Souplet~\cite{PolQuiSou:07} which we recall:

\begin{lemma}\label{doubl} 
Let $(X,d)$ be a complete metric space, $\Gamma \varsubsetneqq X$ and $\gamma: X \setminus \Gamma \to (0, +\infty)$. Assume that $\gamma$ is bounded on all compact subsets of $X \setminus \Gamma$. Given $k>0$, let $y \in X\setminus \Gamma$ be such that 
\begin{equation*}
\gamma(y) \, \dist(y,\Gamma)>2k.
\end{equation*}
Then, there exists $x \in X \setminus \Gamma$ such that 
\begin{itemize}
\item $\gamma(x) \, \dist(x,\Gamma)>2k$;
\item $\gamma(x)\geq \gamma(y)$;
\item $2\gamma(x)\geq \gamma(z)$, $\forall z\in B_{k/\gamma(x)}(x)$.
\end{itemize}
\end{lemma}

\smallskip
\begin{proof}[Proof of Proposition~\ref{apr2}]
To simplify the notation we may assume that $\zeta \equiv 0$.
Assume by contradiction that \eqref{apr4} is false. Then, for every integer $k \geq 1$ there exist $0 < r_k < \frac12\diam{\Omega}$, a solution $u_{k}$ of \eqref{apr3} with $r = r_k$, and $y_k \in \Gw\cap (B_{2r_k}\setminus \overline B_{r_k})$ such that 
$$
u_k(y_k) > (2k)^{\frac{2}{q-1}} \big[ \dist(y_k, \Gamma_{r_k})  \big]^{- \frac{2}{q-1}}.
$$
Applying the previous lemma with 
\[
X = \overline{\Omega} \cap (\overline{B}_{2r_k} \setminus B_{r_k}) \quad \text{and} \quad \gamma = u_k^{\frac{q-1}{2}},
\]
one finds $x_k \in X \setminus \Gamma_{r_k}$ such that
\begin{itemize}
\item[$(i)$] $u_k(x_k) > (2k)^{\frac{2}{q-1}} \big[ \dist(x_k, \Gamma_{r_k})  \big]^{- \frac{2}{q-1}}$;
\item[$(ii)$] $u_k(x_k) \geq u_k(y_k)$;
\item[$(iii)$] $2^{\frac{2}{q-1}} u_k(x_k) \geq u_k(z)$, $\forall z\in B_{R_k}(x_k) \cap \Omega$, with $R_k = k[u_k(x_k)]^{-\frac{q-1}{2}}$.
\end{itemize}

By $(i)$ we have $R_k < \frac{1}{2} \dist(x_k, \Gamma_{r_k})$ and thus
\begin{equation*}
B_{R_k}(x_k) \cap \Gamma_{r_k} = \emptyset.
\end{equation*}
Since $\dist(x_k, \Gamma_{r_k}) \le \frac{1}{2}r_k < \frac{1}{4} \diam{\Omega}$, we also deduce from $(i)$ that
\[
u_k(x_k) \ge \left( \frac{8k}{\diam{\Omega}} \right)^\frac{2}{q-1}.
\]
In particular,
$$
u_k(x_k) \to + \infty \quad \text{as } k \to +\infty.
$$
For every $k \ge 1$, let
\begin{align*}
t_k & = [u_k(x_k)]^{-\frac{q-1}{2}},\\
D_k & = \Big\{ \xi \in \RR^N;\ |\xi| \leq k \text{\; and \;} x_k + t_k \xi \in \Omega \Big\}
\intertext{and}
v_k(\xi) & = \frac{1}{u_k(x_k)} \, u_k \big( x_k + t_k \xi  \big) \quad  \forall \xi \in D_k.
\end{align*}
Then, $v_k$ satisfies
\begin{equation*}
-\Gd v_k = v_k^q,  \quad 0 \leq v_k  \leq 2^{\frac{2}{q-1}} \quad \text{and} \quad v_k(0) = 1.
\end{equation*}

Passing to a subsequence if necessary, we may assume that either
\begin{itemize}
\item[$(A)$] for every $a > 0$ there exists $k_0 \ge 1$ such that if $k \ge k_0$, then
$B_{at_k}(x_k) \cap \partial\Omega = \emptyset$,
\end{itemize}
or
\begin{itemize}
\item[$(B)$] there exists $a_0 > 0$ such that for every $k \ge 1$,
$B_{a_0 t_k}(x_k) \cap \partial\Omega \neq \emptyset$.
\end{itemize}
Since the sequence $(v_k)$ is uniformly bounded, it follows that $(\Delta v_k)$ is also uniformly bounded. In both cases, by elliptic (interior and boundary) estimates, we have for every $1 < p < +\infty$ and every $s > 0$,
\[
\|v_k\|_{W^{2,p}(D_k \cap B_s)} \le C_{s, p}.
\]
If $(A)$ holds, then up to a subsequence $(v_k)$ converges locally uniformly in $\RR^N$ to some smooth function $v$ such that
\begin{equation*}
-\Gd v = v^q,  \quad 0 \leq v  \leq 2^{\frac{2}{q-1}} \quad \text{and} \quad v(0) = 1.
\end{equation*}
On the other hand, if $(B)$ holds, then up to a subsequence and a rotation of the domain there exists some smooth function $v$ defined in $\RR^N_+$  such that $(v_k)$ converges locally uniformly to $v$. Since the sequence $(v_k)$ is equicontinuous and for every $k \ge 1$, $v_k(0) = 0$, we have $v(0) = 1$.

In both cases, we deduce that $v$ is a nontrivial bounded solution of 
\[
-\Gd v = v^q
\]
in $\RR^N$ or in $\RR^N_+$, which is impossible (see \cite{Dan:92}). Therefore, estimate \eqref{apr4} must hold.
\end{proof}

\begin{proof}[Proof of Theorem~\ref{thm1.1}.] 
It suffices to establish \eqref{E5} if $x \in \Omega$ and $|x| < \frac34 \diam{\Omega}$. For this purpose, we apply Proposition~\ref{apr2} with $r = \frac23 |x|$.  Since $\dist(x, \Gamma_r) = \frac{1}{3} r$, we deduce that 
\begin{equation*}
u(x)\leq C \big[ \dist(x, \Gamma_r)  \big]^{- \frac{2}{q-1}} =  C \left(\frac{r}{3} \right)^{- \frac{2}{q-1}} = \widetilde C \, |x|^{- \frac{2}{q-1}}.
\end{equation*}
This establishes the result.
\end{proof}


\section{The geometric and analytic framework}\label{sec5}

We recall some of the preliminaries and the geometric framework in \cite{GmiVer:91} which will be used in the remaining of the paper.

\medskip  
We denote by $(x_{1}, \ldots, x_{N})$ the coordinates of $x \in \RR^N$ and by $\CB=\{\bs \e_{1},\ldots, \bs \e_{N}\}$ the canonical orthonormal basis in  $\BBR^N$. 
Since we are assuming that the outward unit normal vector is $-\bs \e_{N}$, 
 $\partial\Omega$ is the graph of a smooth function in a neighborhood of $0$. In other words,
there exist a neighborhood $G$ of $0$ and a smooth function $\gf : G\cap T_{0}\Gw \to \RR$ such that 
$$
G\cap\prt\Gw= \Big\{(x',x_{N}) \in \RR^{N-1} \times \RR \ ; \ x'\in G\cap T_{0}\Gw \text{ and } x_{N}=\gf(x') \Big\}.
$$
Furthermore, 
\[
\gf(0)= 0 \quad \text{and} \quad \nabla\gf(0)=0.
\]
Setting $\Gf (x)=y$, with $y_{i}=x_{i}$ if $i=1,\ldots, N-1$ and $y_{N}=x_{N}-\gf(x')$, we can assume that $\Gf$ is a $C^\infty$ diffeomorphism from $G$ to $\tilde G=\Gf(G)$, and $\Gf(\Gw\cap G)=\tilde G\cap\BBR_{+}^N$. 
To avoid introducing some additional notation, we will assume that 
\[
\tilde G = B_1.
\]

\medskip
Given $\zeta \in C^\infty(\partial\Omega)$, let $z$ be the harmonic extension of $\gz$ in $\Gw$. For every solution $u$ of \eqref{E1}, we denote
\begin{equation*}
u(x)-z(x)=\tilde u(y),\quad z(x)=\tilde z(y) \quad \text{and} \quad \gz(x)=\tilde\gz(y),
\end{equation*}
for every $x=\Gf^{-1}(y)$ with $y\in \tilde G\cap\BBR_{+}^N$. 
Since $u$ is superharmonic and $u = z$ on $\partial\Omega$, we have $\tilde u\geq 0$. 
On the other hand, a straightforward computation yields
$$
\Delta u = \Gd\tilde u + \abs{\nabla\gf}^2\tilde u_{y_{N},y_{N}} - 2\langle\nabla\gf,\nabla\tilde u_{y_{N}}\rangle - \tilde u_{y_{N}}\Gd\gf
$$
Thus, $\tilde u$ satisfies the equation 
\begin{equation*}
-\Gd\tilde u-\abs{\nabla\gf}^2\tilde u_{y_{N},y_{N}}+2\langle\nabla\gf,\nabla\tilde u_{y_{N}}\rangle+\tilde u_{y_{N}}\Gd\gf
 = (\tilde u+\tilde z)^q.
\end{equation*}
Rewriting this equation in terms of spherical coordinates, one obtains
\begin{multline*}
\left( 1  + \eta_1 \right) \tilde u_{rr} + \frac{1}{r^2} \Gd'\tilde u + \frac{N-1 + \eta_2}{r} \, \tilde u_{r} + (\tilde u+\tilde z)^q = \\
= \frac{1}{r^2} \langle\nabla' \tilde u,  \overrightarrow{\eta_3} \rangle + \frac{1}{r} \langle\nabla' \tilde u_r, \overrightarrow{\eta_4} \rangle 
+ \frac{1}{r^2} \langle\nabla'\langle\nabla'\tilde u, {\bf e}_N\rangle, \overrightarrow{\eta_5} \rangle.
\end{multline*}
where
\begin{equation*}
\begin{aligned}
\eta_1 & = - 2\gf_r\langle{\bf n}, {\bf e}_N\rangle+\abs{\nabla\gf}^2\langle{\bf n},{\bf e}_N\rangle^2,\\
\eta_2 & = -r\langle{\bf n},{\bf e}_N\rangle\Gd\gf-2
\langle\nabla'\langle{\bf n},{\bf e}_N\rangle,\nabla'\gf\rangle
+r\abs{\nabla\gf}^2\langle\nabla'\langle{\bf n},{\bf e}_N\rangle,{\bf e}_N\rangle,\\
\overrightarrow{\eta_3} & = - \left(2\gf_r-\abs{\nabla \gf}^2\langle{\bf n},{\bf e}_N\rangle-r\Gd\gf\right) {\bf e}_N, \\
\overrightarrow{\eta_4} & = - \left( \abs{\nabla \gf}^2\langle{\bf n},{\bf e}_N\rangle - 2\gf_r\right){\bf e}_N
+ \frac{2}{r}  \langle{\bf n},{\bf e}_N \rangle \nabla'\gf,\\
\overrightarrow{\eta_5} & = - \abs{\nabla\gf}^2 {\bf e}_N + \frac{2}{r} \nabla'\gf.
\end{aligned}
\end{equation*}
Taking into account the fact that $\phi(0) = 0$ and $\nabla\phi(0)= 0$,
$$
\abs{\gf(x)}\leq Cr^2,\quad \abs{D\gf(x)} \leq Cr \quad \text{and} \quad \abs{D^2\gf}\leq C.
$$
Thus, for every $j = 1, \ldots, 5$,
\begin{equation*}
\| \eta_j(r, \cdot) \|_{L^\infty} \leq C r \quad \forall r \in (0, 1).
\end{equation*}

\begin{lemma}\label{lemma5.0}
Let
\begin{equation}\label{4.2}
t= \log{\tfrac{1}{r}}, \quad v(t,\gs) = r^{\frac{2}{q-1}} \, \tilde u(r,\gs)  \quad \text{and} \quad \alpha(t,\gs) = r^{\frac{2}{q-1}} \, \tilde z(r,\gs).
\end{equation}
Then, $v$ satisfies
\begin{multline}\label{lift4}
\left(1+\epsilon_1\right)v_{tt} + \Gd'v - \left(N- \tfrac{2(q+1)}{q-1}+\epsilon_2\right)v_{t}
+\left(\ell_{N,q}+\epsilon_3\right)v+ (v+\ga)^q = \\
=\langle\nabla'v,\overrightarrow {\epsilon_4}\rangle 
+\langle\nabla'v_t,\overrightarrow {\epsilon_5}\rangle+
\langle\nabla'\langle \nabla' v,{\bf e}_N\rangle,\overrightarrow {\epsilon_6}\rangle,
\end{multline}
where $\epsilon_j$ are functions defined in $(0, +\infty) \times S_+^{N-1}$ satisfying the estimates
\begin{equation}\label{4.3}
\| \epsilon_j(t, \cdot) \|_{L^\infty} \leq C \e^{-t} \quad \forall t \geq 0,
\end{equation}
for every $j = 1, \ldots, 6$.
\end{lemma}

We refer the reader to \cite{GmiVer:91} for the proof of Lemma~\ref{lemma5.0} and for the explicit expressions of the functions $\epsilon_j$.

\bigskip

For every $T \geq 0$ and $\delta > 0$, let
\begin{equation*}
Q_T = (T, +\infty) \times S_+^{N-1} \quad \text{and} \quad Q_{T,\delta} = (T-\delta, T+\delta) \times S_+^{N-1}.
\end{equation*}

We have the following $W^{2,p}$-estimates satisfied by $v$:

\begin{prop}\label{prop5.1}
Let $v$ be defined as in Lemma~\ref{lemma5.0}. If $v$ is uniformly bounded in $Q_0$, then for every $1 < p < +\infty$,
\begin{equation}\label{5.3}
\|v\|_{W^{2,p}(Q_{T,1})} \leq C \Big( \|v\|_{L^2(Q_{T,2})} + \e^{- \frac{2T}{q-1}} \Big) \quad \forall T \geq 2,
\end{equation}
for some positive constant depending on $\|v\|_{L^\infty}$ and on $p$.
\end{prop}

\begin{proof}
Since $\Delta'$ is uniformly elliptic and $\Phi$ is a diffeomorphism, the operator $L$ given by
\begin{multline*}
L(v) = \left(1+\epsilon_1\right)v_{tt} + \Gd'v - \left(N- \tfrac{2(q+1)}{q-1}+\epsilon_2\right)v_{t} +
 \\
- \langle\nabla'v,\overrightarrow {\epsilon_4}\rangle 
- \langle\nabla'v_t,\overrightarrow {\epsilon_5}\rangle -
\langle\nabla'\langle \nabla' v,{\bf e}_N\rangle,\overrightarrow {\epsilon_6}\rangle
\end{multline*}
is uniformly elliptic. Let $\delta > 0$. By the Agmon-Douglis-Nirenberg estimates (see \cite{AgmDouNir:59}) applied to the restriction of $v$ on the set $Q_{T, 1 + \delta}$,
\[
\|v\|_{W^{2,p}(Q_{T,1 + \delta})} \leq C \Big( \|v\|_{L^p(Q_{T, 1 + 2\delta})} + \|(\alpha + v)^q\|_{L^p(Q_{T,1 + 2\delta})} \Big).
\]
Since $\alpha$ and $v$ are uniformly bounded in $Q_0$, for every $s \in (1, 2)$ we have
\[
\begin{split}
\|(\alpha + v)^q\|_{L^p(Q_{T,s})} 
& \le  \|\alpha + v\|_{L^\infty(Q_{T,s})}^{q-1} \|\alpha + v\|_{L^p(Q_{T,s})} \\
& \le C \Big( \|\alpha \|_{L^p(Q_{T,s})} + \| v\|_{L^p(Q_{T,s})} \Big).
\end{split}
\]
Since $\tilde z$ is uniformly bounded in $\Omega$,
\[
\|\alpha \|_{L^p(Q_{T,s})} \le C \e^{- \frac{2T}{q-1}}\|\tilde z\|_{L^\infty(\Omega)} \le C \e^{- \frac{2T}{q-1}}.
\]
Thus,
\begin{equation}\label{5.4}
\|v\|_{W^{2,p}(Q_{T,1+\delta})} \leq C \Big( \|v\|_{L^p(Q_{T,1 + 2\delta})} + \e^{- \frac{2T}{q-1}} \Big).
\end{equation}
In particular,
\[
\|v\|_{W^{2,p}(Q_{T,1})} \leq C \Big( \|v\|_{L^p(Q_{T, \frac{3}{2})})} + \e^{- \frac{2T}{q-1}} \Big).
\]
By a bootstrap argument based on the estimate \eqref{5.4} above and the Sobolev imbedding, we also have
\[
\|v\|_{L^{p}(Q_{T, \frac{3}{2}})} \leq C \Big( \|v\|_{L^2(Q_{T,2})} + \e^{- \frac{2T}{q-1}} \Big).
\]
Combining these inequalities, the estimate follows.
\end{proof}


\section{Removable singularities at $0$}\label{sec6}

The goal of this section is to show that solutions of \eqref{E1} which are not too large in a neighborhood of $0$ must be continuous at $0$.

\begin{thm}\label{prop6.1} 
Let $q > q_{1}$ and let $u$ be a solution of \eqref{E1}. If
\begin{equation}\label{6.1}
\lim_{x \to 0}{\abs x^{\frac{2}{q-1}} u(x) } = 0,
\end{equation}
then $u$ can be continuously extended at $0$.
\end{thm}

\begin{proof} 
Let $v$ be the function given by \eqref{4.2}. By assumption \eqref{6.1}, we have
\begin{equation}\label{6.2}
\lim_{t \to +\infty}{v(t,\cdot)} = 0 \quad \text{uniformly in } S_+^{N-1}.
\end{equation}
We now rewrite \eqref{lift4} under the form
\begin{equation}\label{new1}
v_{tt} - \left(N-\tfrac{2(q+1)}{q-1}\right)v_{t}
+\ell_{N,q}v+\Gd'v+(v+\ga)^q= H,
\end{equation}
where $H$ is given by
\begin{equation}\label{new2}
H= -\epsilon_{1}v_{tt} + \epsilon_{2}v_{t} - \epsilon_{3}v + \langle\nabla'v,\overrightarrow {\epsilon_4}\rangle
+\langle\nabla'v_t,\overrightarrow {\epsilon_5}\rangle+
\langle\nabla'\langle \nabla' v,{\bf e}_N\rangle,\overrightarrow {\epsilon_6}\rangle.
\end{equation}
Thus,
\begin{multline}\label{6.3}
\int\limits_{S^{N-1}_{+}} v v_{tt} \, d\sigma - \left(N-\tfrac{2(q+1)}{q-1}\right) \int\limits_{S^{N-1}_{+}} v v_{t} \, d\sigma  +\ell_{N,q} \int\limits_{S^{N-1}_{+}} v^2  \, d\sigma + \int\limits_{S^{N-1}_{+}} v\Gd'v \, d\sigma + \\
+ \int\limits_{S^{N-1}_{+}} v(v+\ga)^q \, d\sigma = \int\limits_{S^{N-1}_{+}} v H d\sigma.
\end{multline}
Let 
$$
X(t)= \|v(t, \cdot)\|_{L^2(S^{N-1}_{+})} \quad \forall t \geq 0.
$$
Note that for every $t > 0$,
\begin{equation}\label{6.5}
X X_t = \int\limits_{S^{N-1}_{+}} vv_{t} \,d\sigma.
\end{equation}
Using H\"older's inequality we have
\[
|X X_t| \le \|v(t, \cdot)\|_{L^2(S^{N-1}_{+})}\|v_t(t, \cdot)\|_{L^2(S^{N-1}_{+})}.
\]
Thus,
\begin{equation}\label{6.5a}
|X_t| \le \|v_t(t, \cdot)\|_{L^2(S^{N-1}_{+})}.
\end{equation}
Computing the derivative with respect to $t$ on both sides of identity \eqref{6.5}, we get
\[
(X_t)^2 + X X_{tt} = \int\limits_{S^{N-1}_{+}} (v_{t})^2 \,d\sigma + \int\limits_{S^{N-1}_{+}} v v_{tt} \,d\sigma = \|v_t(t, \cdot)\|_{L^2(S^{N-1}_{+})}^2 + \int\limits_{S^{N-1}_{+}} v v_{tt} \,d\sigma.
\]
From this identity and estimate \eqref{6.5a}, we deduce that
\begin{equation}\label{6.5b}
X X_{tt} \geq \int\limits_{S^{N-1}_{+}} v v_{tt} \,d\sigma.
\end{equation}
On the other hand, since the first eigenvalue of the Laplace-Beltrami operator $-\Gd'$ in $W^{1,2}_{0}(S^{N-1}_{+})$ is $N-1$, 
\begin{equation*}
(N-1) X^2 \leq \int\limits_{S^{N-1}_{+}} |\nabla' v|^2 \, d\sigma = - \int\limits_{S^{N-1}_{+}} v \, \Gd' v \, d\gs.
\end{equation*}
By H\"older's inequality,
\begin{equation*}
\int\limits_{S^{N-1}_{+}} v H \, d\sigma \leq  X \|H(t,\cdot)\|_{L^2(S^{N-1}_{+})}.
\end{equation*}
From the elementary inequality
\[
(v+\ga)^q \leq 2^q(v^q+\ga^q),
\]
we get
\[
\int\limits_{S^{N-1}_{+}} v (v+\ga)^q \, d\sigma  \leq 2^q \int\limits_{S^{N-1}_{+}} \big(v^{q+1} + v \ga^q \big) \, d\sigma 
\]
It follows from H\"older's inequality that
\begin{equation}\label{6.8}
\int\limits_{S^{N-1}_{+}} v (v+\ga)^q \, d\sigma  \leq 2^q \Big( X^2 \|v(t, \cdot)\|_{L^\infty(S^{N-1}_+)}^{q-1} + X \|\alpha(t, \cdot)\|_{L^{2q}(S^{N-1}_+)}^q \Big). 
\end{equation}
We may assume that $u$ is a nontrivial solution of \eqref{E1}.
By the strong maximum principle, we have $u > 0$ in $\Omega$, thus $X > 0$. Combining \eqref{6.3}, \eqref{6.5} and \eqref{6.5b}--\eqref{6.8}, one gets
\begin{multline*}
X_{tt} - \left(N-\tfrac{2(q+1)}{q-1}\right)X_{t}
+ \Big(\ell_{N,q} - N + 1 + 2^q \|v(t, \cdot)\|_{L^\infty}^{q-1} \Big)X \geq\\
\geq - \big( \|H(t,\cdot)\|_{L^2} + 2^q \|\alpha(t, \cdot)\|_{L^{2q}}^q \big)
\end{multline*}
(to simplify the notation we drop the explicit dependence of the set $S_+^{N-1}$).
From the definition of the function $\alpha$, there exists $C > 0$ such that
\[
2^q \|\alpha(t, \cdot)\|_{L^{2q}}^q \le C \e^{- \frac{2qt}{q-1}}.
\]
In view of \eqref{6.2}, given $\eps > 0$ there exists $t_0 > 0$ such that
\[
2^q \|v(t, \cdot)\|_{L^\infty}^{q-1}  \leq \eps \quad \text{on $[t_0, \infty)$.}
\]
We deduce that for every $t \geq t_0$ we have
\begin{equation*}
X_{tt} - \left(N-\tfrac{2(q+1)}{q-1}\right)X_{t}
+ \left(\ell_{N,q} - N + 1 + \eps \right)X \geq - \|H(t,\cdot)\|_{L^2} - C\e^{- \frac{2qt}{q-1}}.
\end{equation*}
We shall show that
\begin{equation*}
X(t) \leq C \e^{- \frac{2t}{q-1}} \quad \forall t \geq 0,
\end{equation*}
and the conclusion will now follow from a bootstrap argument. Note that the linear equation 
\begin{equation*}
Z_{tt} - \left(N-\tfrac{2(q+1)}{q-1}\right)Z_{t}
+ \left(\ell_{N,q} - N + 1 \right)Z  = 0
\end{equation*}
has two linearly independent solutions:
\[
Z_1(t) =  \e^{- \frac{q+1}{q-1} t} \quad \text{and} \quad Z_2(t) = \e^{(N - \frac{q+1}{q-1}) t}.
\]
We can then take $\eps>0$ small enough so that the linear equation 
\begin{equation*}
Z_{tt} - \left(N-\tfrac{2(q+1)}{q-1}\right)Z_{t}
+ \left(\ell_{N,q} - N + 1 + \eps \right)Z  = 0
\end{equation*}
has two linearly independent solutions:
\[
Z_{1, \eps}(t) =  \e^{r_{1, \eps} t} \quad \text{and} \quad Z_{1,\eps}(t) = \e^{r_{2, \eps} t}
\]
such that
\[
r_{1, \eps} < - \frac{2}{q-1}  \quad \text{and} \quad r_{2, \eps} > 0.
\]
In particular,
\[
Z_{2, \eps}(t) \to +\infty \quad \text{as } t \to +\infty.
\]
From assumption \eqref{6.1}, $v$ is bounded. In view of \eqref{4.3} and Proposition~\ref{prop5.1} with $p = 2$, there exists $C_1 > 0$ such that
\begin{equation*}
\|H(t,\cdot)\|_{L^2} \leq C_1 \e^{-t} \quad \forall t \ge 0.
\end{equation*}
Thus,
\begin{equation*}
X_{tt} - \left(N-\tfrac{2(q+1)}{q-1}\right)X_{t}
+ \left(\ell_{N,q} - N + 1 \right)X \geq - \hat C_1 \e^{-t}. 
\end{equation*}
Since 
\[
X(t) \to 0  \quad \text{as } t \to +\infty,
\]
from the maximum principle there exists a constant $\tilde C_1 > 0$ such that
\[
X(t) \le \tilde C_1 (Z_{1, \eps}(t) + \e^{-t}).
\]
If $r_{1, \eps} \ge -1$, then 
\[
X(t) \le 2 \tilde C_1 Z_{1, \eps}(t).
\]
Since $r_{1, \eps} < - \frac{2}{q-1}$, the estimate above implies that $u$ is bounded and thus by standard elliptic estimates $u$ is continuous. Otherwise $r_{1, \eps} < -1$, in which case, 
\[
X(t) \le 2 \tilde C_1 \e^{-t}.
\]
Thus, by Proposition~\ref{prop6.1} for every $T \ge 2$,
\[
\|v\|_{W^{2,2}(Q_{T, 2})} \le \widetilde C_1 \e^{-T}.
\]
In view of \eqref{4.3}, there exists $C_2 > 0$ such that
\begin{equation*}
\|H(t,\cdot)\|_{L^2} \leq C_2 \e^{-2t} \quad \forall t \ge 0.
\end{equation*}
Thus,
\begin{equation*}
X_{tt} - \left(N-\tfrac{2(q+1)}{q-1}\right)X_{t}
+ \left(\ell_{N,q} - N + 1 \right)X \geq - \hat C_2 \e^{-2t}. 
\end{equation*}
This implies as before that
\[
X(t) \le \tilde C_2 (Z_{1, \eps}(t) + \e^{-2t}).
\]
If $r_{1, \eps} \ge -2$, then 
\[
X(t) \le 2 \tilde C_2 Z_{1, \eps}(t)
\]
and $u$ is bounded. Otherwise $r_{1, \eps} < -2$, in which case, 
\[
X(t) \le 2 \tilde C_2 \, \e^{-2t}.
\]
We can continue this argument and deduce in finitely many steps that 
\[
X(t) \le 2 \tilde C_k Z_{1, \eps}(t).
\]
Applying Proposition~\ref{prop5.1} with $p > \frac{N}{2}$, we deduce that for every $T \ge 2$,
\[
\|v\|_{W^{2,p}(Q_{T,1})} \le C \Big( Z_{1, \eps}(T) + \e^{-\frac{2T}{q-1}} \Big) \le C \, \e^{-\frac{2T}{q-1}}.
\]
Thus, by Morrey's embedding,
\[
\|v\|_{L^\infty(Q_{T,1})} \le C \, \e^{-\frac{2T}{q-1}}.
\]
This implies that $u$ is bounded and hence continuous in $\overline{\Omega}$.
\end{proof}

The conclusion of Theorem~\ref{prop6.1} is false with the critical exponent $q = q_1$. In fact, combining Theorem~\ref{thm1.2a} and the result of del Pino-Musso-Pacard mentioned in the Introduction (Theorem~\ref{thm1.3}), when $q= q_1$ there exist solutions of \eqref{E1} such that
\[
u(x) \sim x_N |x|^{-N} \big(\log{\tfrac{1}{|x|}} \big)^{-\frac{N-1}{2}} 
\] 
in a neighborhood of $0$. These solutions are necessarily discontinuous at $0$ but, since $\frac{2}{q_1-1} = N-1$,
\[
\lim_{x \to 0}{ |x|^{\frac{2}{q_1-1}} u(x)} = 0.
\]

\smallskip

The right statement in this case is the following:

\begin{thm}\label{prop6.2} 
Let $q = q_{1}$ and let $u$ be a nonnegative solution of \eqref{E1}. If
\begin{equation*}
\lim_{x \to 0}{ |x|^{N-1} \big(\log{\tfrac{1}{|x|}} \big)^{\frac{N-1}{2}} u(x)} = 0,
\end{equation*}
then $u$ can be continuously extended at $0$.
\end{thm}

\begin{proof}
Let 
$$
W(t)= t^{\frac{N-1}{2}} \|v(t, \cdot)\|_{L^2(S^{N-1}_{+})} \quad \forall t \geq 0,
$$
where $v$ is the function given by \eqref{4.2}. By assumption, $W(t) \to 0$ as $t \to +\infty$. As in the proof of Theorem~\ref{prop6.1}, for any $\eps>0$ there exists $t_0 > 0$ such that for every $t \ge t_0$,
\begin{multline*}
W_{tt} + \left(N - \tfrac{N-1}{t} \right)W_{t}
+ \frac{1}{t} \left(-\tfrac{N(N-1)}{2}  + \eps + \tfrac{N^2-1}{4t} \right) W \geq \\
\geq - t^{\frac{N-1}{2}} \|H(t,\cdot)\|_{L^2} - C\,t^{\frac{N-1}{2}} \e^{- (N+1) t}.
\end{multline*}
The linear equation 
\begin{equation*}
W_{tt} + \left(N - \tfrac{N-1}{t} \right)W_{t}
+ \frac{1}{t} \left(- \tfrac{N(N-1)}{2} + \tfrac{N^2-1}{4t} \right) W = 0
\end{equation*}
has two linearly independent solutions $W_1$ and $W_2$ such that for $t$ sufficiently large (see Lemma~\ref{lemma2.3} below)
\[
W_1(t) = t^{\frac{N-1}{2}} \e^{- N t} (1 + o(1)) \quad \text{and} \quad W_2(t) = t^{\frac{N-1}{2}} (1 + o(1)).
\]
We can then take $\eps>0$ small enough so that the linear equation 
\begin{equation*}
W_{tt} + \left(N - \tfrac{N-1}{t} \right)W_{t}
+ \frac{1}{t} \left(\tfrac{ - N(N-1)}{2} + \eps  + \tfrac{N^2-1}{4t} \right) W = 0
\end{equation*}
has two linearly independent solutions $W_{1, \eps}$ and $W_{2, \eps}$ such that
\[
W_{1, \eps}(t) \le C t^{\frac{N-1}{2}} \e^{- (N-1) t}
\]
and
\[
W_{2, \eps}(t) \to +\infty \quad \text{as } t \to +\infty.
\]
In view of \eqref{4.3} and Proposition~\ref{prop5.1} with $p = 2$, there exists $C_1 > 0$ such that
\begin{equation*}
\|H(t,\cdot)\|_{L^2} \leq C_1 t^{-\frac{N-1}{2}}\e^{-t} \quad \forall t > 0.
\end{equation*}
Thus,
\begin{equation*}
W_{tt} + \left(N - \tfrac{N-1}{t} \right)W_{t}
+ \frac{1}{t} \left(-\tfrac{N(N-1)}{2} + \eps + \tfrac{N^2-1}{4t} \right)W \geq - C \e^{- t}.
\end{equation*}
Since 
\[
W(t) \to 0  \quad \text{as } t \to +\infty,
\]
from the maximum principle there exists a constant $\tilde C_1 > 0$ such that
\[
W(t) \le \tilde C_1 (W_{1, \eps}(t) + \e^{-t}).
\]
Thus, 
\[
W(t) \le \hat C_{1} \e^{-t}.
\]
Thus, by Proposition~\ref{prop5.1} with $p=2$, for every $T \ge 2$,
\[
\|v\|_{W^{2,2}(Q_{T, 2})} \le \widehat C_1 t^{-\frac{N-1}{2}}\e^{-T}.
\]
In view of \eqref{4.3}, there exists $\hat C_2 > 0$ such that
\begin{equation*}
\|H(t,\cdot)\|_{L^2} \leq \hat C_2 t^{-\frac{N-1}{2}}\e^{-2t} \quad \forall t \ge 0.
\end{equation*}
We can continue this argument as in the previous theorem and deduce after finitely many steps that 
\[
W(t) \le \hat C_{k} W_{1, \eps}(t) \le  \tilde C \, t^{\frac{N-1}{2}} \e^{- (N-1) t},
\]
which implies that $u$ is bounded and hence continuous in $\overline{\Omega}$.
\end{proof}


\section{Proof of Theorem~\ref{thm1.2}}\label{sec4}

We first establish the following

\begin{prop}\label{cv1} 
Let $q_1 \leq q < q_3$, with $q \neq q_2$. If $u$ is a solution of \eqref{E1} such that
\begin{equation*}
\abs x^{\frac{2}{q-1}}u(x) \; \text{ is bounded in } \Omega,
\end{equation*}
then for every $\eps > 0$ there exists $\delta > 0$ such that if $x \in \Omega \setminus \{0\}$, $\frac{x}{|x|} \in S_+^{N-1}$ and $|x| < \delta$, then
\begin{equation}\label{4.1}
\Big| |x|^{\frac{2}{q-1}} u(x) - w\big(\tfrac{x}{|x|}\big) \Big| < \eps
\end{equation}
where $w$ is a solution of \eqref{En3}.
\end{prop}


\begin{proof} 
Let $v$ be the function given by \eqref{4.2}. We first rewrite equation \eqref{lift4} under the form
\begin{equation}\label{new1x}
v_{tt} 
+\ell_{N,q}v+\Gd'v+(v+\ga)^q - \left(N-\tfrac{2(q+1)}{q-1}\right)v_{t} = H,
\end{equation}
where $H$ is given by \eqref{new2}.
Multiplying \eqref{new1x} by $v_{t}$ and integrating over $S^{N-1}_{+}$ yields
\begin{multline*}
\int\limits_{S^{N-1}_{+}} v_t v_{tt} \, d\sigma +\ell_{N,q} \int\limits_{S^{N-1}_{+}} v_t v  \, d\sigma + \int\limits_{S^{N-1}_{+}} v_t \, \Gd'v \, d\sigma 
+ \int\limits_{S^{N-1}_{+}} v_t (v+\ga)^q \, d\sigma + \\ - \left(N-\tfrac{2(q+1)}{q-1}\right) \int\limits_{S^{N-1}_{+}} (v_{t})^2 \, d\sigma  = \int\limits_{S^{N-1}_{+}}  v_t H \, d\sigma.
\end{multline*}
Thus,
\begin{multline}\label{f2}
\frac{d}{dt}\int\limits_{S^{N-1}_{+}} \bigg[\frac{(v_{t})^2}{2}  + \frac{\ell_{N,q} v^2}{2} - \frac{\abs{\nabla'v}^2}{2} + \frac{(v+\ga)^{q+1}}{q+1} \bigg] \, d\gs - \left(N-\tfrac{2(q+1)}{q-1}\right) \int\limits_{S^{N-1}_{+}} (v_{t})^2 \, d\sigma  = \\
= \int\limits_{S^{N-1}_{+}} \Big[ v_t H + \ga_{t} (v+\ga)^{q} \Big] \, d\sigma.
\end{multline}
From our assumption on $u$, $v$ is bounded. It follows from \eqref{5.3} and the Sobolev imbedding that $v$, $v_t$ and $\nabla' v$ are uniformly bounded in $S^{N-1}_+ \times \RR_+$. Integrating \eqref{f2} from $0$ to $T$, for any $T > 0$,
one deduces that
$$
\int\limits_{S^{N-1}_{+}} \bigg|\frac{(v_{t})^2}{2}  + \frac{\ell_{N,q} v^2}{2} - \frac{\abs{\nabla'v}^2}{2} + \frac{(v+\ga)^{q+1}}{q+1} \bigg| \, d\gs \leq C \quad \text{in } \RR_+
$$
for some constant $C > 0$. On the other hand, 
$$
\int\limits_{S_+^{N-1}} |v_{t} H| \, d\sigma \leq C \e^{-t}.
$$
Moreover, since $v$ is bounded and $\alpha$ satisfies \eqref{5.4}, we have
$$
\int\limits_{S^{N-1}_{+}} |\ga_{t}| (v+\ga)^{q}  \, d\sigma \leq C \e^{- \frac{2t}{q-1}}.
$$
Thus, integrating \eqref{f2} on $(0, +\infty)$, we obtain
\begin{equation*}
\left|N-\tfrac{2(q+1)}{q-1}\right| \int_{0}^{+\infty} \!\!\!\! \int\limits_{S^{N-1}_{+}}
v_{t}^2 \, d\gs< + \infty.
\end{equation*}
Since $q \neq q_{2}$, $N-\frac{2(q+1)}{q-1} \neq 0$. Hence,
\[
\int_{0}^{+\infty} \!\!\!\! \int\limits_{S^{N-1}_{+}}
v_{t}^2 \, d\gs< +\infty.
\]
By \eqref{5.3} and Morrey's estimates, $v_t$ is uniformly continuous on $Q_0$. We deduce that
\[
v_t(t, \cdot) \to 0 \quad \text{uniformly in $S_+^{N-1}$ as } t \to +\infty.
\]

We now prove that
\[
v(t,\cdot) \to w \quad \text{uniformly in $S_+^{N-1}$ as } t \to +\infty,
\]
where $w$ is a nonnegative solution of \eqref{En3}. For this purpose, we study the limit set of the trajectories of $v$, namely the set
$$
\Gg=\bigcap_{\gt > 0}\overline{\bigcup_{t \geq \gt}\{v(t,.)\}},
$$ 
where the closure is computed with respect to the usual norm in $C^0(S_+^{N-1})$. Since $\Gamma$ is the intersection of a decreasing family of closed connected subsets of $C^0(S_+^{N-1})$, $\Gamma$ is closed and connected. In addition, since $v$ is uniformly continuous in $Q_0$, it follows from the Arzel\`a-Ascoli theorem that $\Gamma$ is also compact and nonnempty.

We claim that every $w \in \Gamma$ satisfies problem \eqref{En3}. Indeed, let $(t_k)$ be a sequence of nonnegative real numbers such that $t_k \to +\infty$ and
$$
v(t_k, \cdot) \to w \quad \text{uniformly in } S_+^{N-1}.
$$
Clearly, $w$ is nonnegative and $w = 0$ on $\partial S_+^{N-1}$. For each $k \ge 1$, let
$$
V_k : (s, \sigma) \in [0, 1] \times S_+^{N-1} \longmapsto v(t_k + s, \sigma).
$$
For every $\varphi \in C_0^\infty(S_+^{N-1})$ and for every $\eps \in (0, 1)$, from the equation satisfied by $v$ we have
\begin{multline*}
\int_0^\eps \!\! \int\limits_{S^{N-1}_{+}} \Big[ (V_k)_{tt}\varphi + \ell_{N,q} V_k \varphi + V_k  \Delta'\varphi +(V_k+\ga)^q \varphi - \left(N-\tfrac{2(q+1)}{q-1}\right)(V_k)_{t} \varphi \Big] \, d\sigma \, dt =\\
 =  \int_{t_k}^{t_k + \eps} \!\!\!\! \int\limits_{S^{N-1}_{+}}  H \varphi \, d\sigma \, dt.
\end{multline*}
As $k \to +\infty$,
\[
\int_{t_k}^{t_k + \eps} \!\!\!\! \int\limits_{S^{N-1}_{+}}  H \varphi \, d\sigma \, dt\to 0.
\]
Since $v_t \to 0$ uniformly as $t \to +\infty$, we also have
\[
\int_0^\eps \!\! \int\limits_{S^{N-1}_{+}} (V_k)_{t} \varphi \, d\sigma \, dt \to 0.
\]
Note that
\[
\int_{0}^{\eps} \!\!\!\! \int\limits_{S_+^{N-1}} (V_k)_{tt} \varphi \, d\sigma \, d\tau = \int\limits_{S_+^{N-1}} \big[v_{t}(t_k+ \eps, \sigma) - v_{t}(t_k, \sigma) \big] \varphi \, d\sigma \to 0.
\]
Since the sequence $(V_k)$ is bounded in $C^1$, passing to a subsequence if necessary, we may assume that for some continuous function $W$,
$$
V_k \to W  \quad \text{uniformly in } [0, 1] \times S_+^{N-1}.
$$
We conclude that for every $\eps \in (0, 1)$,
\begin{equation*}
\int_0^\eps \!\! \int\limits_{S^{N-1}_{+}} \Big[ \ell_{N,q} W \varphi - W \Delta'\varphi + W^q \varphi \Big] \, d\sigma \, dt = 0.
\end{equation*}
Dividing both sides by $\eps$ and letting $\eps \to 0$, we get
$$
\int\limits_{S^{N-1}_{+}} \Big[ \ell_{N,q} W(0, \sigma) \varphi - W(0, \sigma) \Delta'\varphi + (W(0, \sigma))^q \varphi \Big] \, d\sigma = 0.
$$ 
Since $w = W(0,\cdot)$, we conclude that $w$ satisfies \eqref{En3}.
Hence, every element of $\Gamma$ is a nonnegative solution of \eqref{En3}. Since these solutions form a discrete subset of $C^0(S_+^{N-1})$ and $\Gamma$ is connected (in our case, the set of nonnegative solutions is $\{0, \omega\}$, where $\omega$ is the unique positive solution of \eqref{En3}), $\Gamma$ contains a single element.
In particular,
\[
v(t,\cdot) \to w \quad \text{uniformly in $S_+^{N-1}$ as } t \to +\infty.
\]
The proposition follows from this convergence.
\end{proof}

\begin{proof}[Proof of Theorem~\ref{thm1.2}]
Let $u$ be a solution of \eqref{E1}. Since $q < q_2$, by Theorem~\ref{thm1.1} there exists $C > 0$ such that for every $x \in \Omega$,
\[
0 \le \abs x^{\frac{2}{q-1}} u(x) \le C.
\]
Thus, by Proposition~\ref{cv1}, there exists a solution $w$ of \eqref{En3} such that \eqref{4.1} holds. Either $w$ is the unique positive solution of \eqref{En3} (see Theorem~\ref{thm1.0}) or $w = 0$. If $w = 0$, then
$$
\lim_{x \to 0}{\abs x^{\frac{2}{q-1}} u(x) } = 0.
$$
Hence, by Theorem~\ref{prop6.1} $u$ can be continuously extended at $0$.
\end{proof}


\section{Proof of Theorem~\ref{thm1.2a}}\label{sec7a}

We first prove an estimate which improves Theorem~\ref{thm1.1} when $q= q_1$, except that we do not know whether the constant $C$ below can be chosen independently of the solution.

\begin{thm}\label{prop4.1}
Assume that $q = q_1$. Then, every solution of \eqref{E1} satisfies
\begin{equation*}
u(x) \leq C |x|^{-(N-1)} \big(\log{\tfrac{1}{|x|}} \big)^{-\frac{N-1}{2}} \quad \forall x \in \Omega,
\end{equation*}
for some constant $C>0$ possibly depending on the solution.
\end{thm}

In the proof of this result we need the following lemma:

\begin{lemma}\label{lemma4.1}
Let $a= q_1$ and $E = \ker{[\Delta' + (N-1) I]}$.  Given a solution of \eqref{E1}, denote by $v$ the function given by \eqref{4.2}. If
\[
v = v_1 + v_2
\]
is the decomposition of $v$ as the orthogonal projections in $L^2(S^{N-1}_+)$ onto $E$ and $E^\bot$, respectively, then
\begin{equation}\label{4.6}
\|v_1(t, \cdot)\|_{L^2(S^{N-1}_+)} \leq C \, t^{- \frac{N-1}{2}} \quad \text{and} \quad  \|v_2(t, \cdot)\|_{L^2(S^{N-1}_+)} \leq C \, \e^{-\frac{t}{2}} \quad \forall t > 0.
\end{equation}
\end{lemma}

\begin{proof}
Denoting by $\phi_1$ the first eigenfunction of $\Delta'$ with $\|\phi_1\|_{L^1} =1$, we have
$$
v_1(t,\sigma) = y(t) \phi_1(\sigma) \quad \text{where} \quad y(t)=\int\limits_{S^{N-1}_{+}} v(t,\gs)\phi_{1}(\gs) \, d\gs.
$$
Since $q = q_1$, equation \eqref{new1} becomes
\begin{equation}\label{4.8}
v_{tt} + N v_{t} + (N-1) v + \Gd'v + (v+\ga)^{q_1}= H,
\end{equation}
with $H$ defined in \eqref{new2}. Since $\alpha \geq 0$, we have
$(v+\ga)^{q_1} \geq v^{q_1}$. Thus,
$$
v_{tt} + N v_{t} + (N-1) v + \Gd'v + v^{q_1} \leq H.
$$
By Jensen's inequality, 
$$
y^{q_1} \leq \int\limits_{S_+^{N-1}} v^{q_1} \phi_1 \, d\sigma.
$$
Multiplying \eqref{4.8} by $\phi_1$ and integrating over $S^{N-1}_+$, we get
\begin{equation*}
y'' + Ny'+ y^{q_1} \leq \int\limits_{S^{N-1}_{+}} H \phi_{1} \, d\gs.
\end{equation*}
By Theorem~\ref{thm1.1}, $v$ is uniformly bounded in $\RR_+ \times S_+^{N-1}$. In particular, by \eqref{4.3} and Proposition~\ref{prop5.1} with $p =2$, we have for every $t \ge 0$,
\[
\int\limits_{S^{N-1}_{+}} H \phi_{1} \, d\gs \leq C \, \e^{-t}.
\]
Thus,
\begin{equation*}
y'' + Ny'+ y^{q_1} \leq C \, \e^{-t}.
\end{equation*}
Applying Lemma~\ref{lemma2.1} we deduce that
$$
y(t) \leq  C t^{- \frac{N-1}{2}} \quad \forall t > 0.
$$
This concludes the proof of the first estimate in \eqref{4.6}.

\medskip
In order to prove the estimate for $v_2$, let 
$$
Y(t)=\|v_2(t, \cdot)\|_{L^2(S^{N-1}_+)} \quad \forall t \geq 0.
$$
Since $v(t,\sigma) = y(t) \phi_1(\sigma) + v_2(t,\sigma)$, we have
$$
v_t = y_t \phi_1 + (v_2)_t \quad \text{and} \quad v_{tt} = y_{tt} \phi_1 + (v_2)_{tt}.
$$
Using the orthogonality between $\phi_1$ and $v_2$,
\begin{equation*}
Y Y_t = \int\limits_{S^{N-1}_{+}} v_2 (v_2)_{t} \,d\sigma = \int\limits_{S^{N-1}_{+}} v_2 \big[y_t \phi_1 + (v_2)_{t} \big] \,d\sigma = \int\limits_{S^{N-1}_{+}} v_2 v_t \, d\sigma.
\end{equation*}
From the first equality, we have
\[
|Y_t| \le \|v_2(t, \cdot)\|_{L^2}.
\]
One also shows that
\begin{equation*}
Y Y_{tt} \geq \int\limits_{S^{N-1}_{+}} v_2 v_{tt} \,d\sigma.
\end{equation*}
On the other hand, since the second eigenvalue of the Laplace-Beltrami operator $-\Gd'$ in $W^{1,2}_{0}(S^{N-1}_{+})$ is $2N$, 
\begin{equation*}
2N Y^2 \leq \int\limits_{S^{N-1}_{+}} |\nabla' v_2|^2 \, d\sigma = - \int\limits_{S^{N-1}_{+}} v_2 \, \Gd' v_2 \, d\gs = - \int\limits_{S^{N-1}_{+}} v_2 \, \Gd' v \, d\gs.
\end{equation*}
Multiply \eqref{4.8} by $v_2$ and integrate over $S_+^{N-1}$. As in the proof of Theorem~\ref{prop6.1}, for every $\eps >0$ there exists $t_1 > 0$ such that for every $t \ge t_1$,
\begin{equation*}
Y_{tt} + N Y_{t} - (N+1-\eps) Y \geq - C \, \e^{-t}.
\end{equation*}
Note that for $\eps > 0$ small the linear equation
\[
Z_{tt} + N Z_{t} - (N+1-\eps) Z = 0
\]
has two linearly independent solutions $Z_{1, \eps}$ and $Z_{2, \eps}$ such that
\[
Z_{1, \eps}(t) = \e^{r_{1, \eps}t} \quad \text{and} \quad Z_{2, \eps}(t) = \e^{r_{2, \eps}t} 
\]
with
\[
r_{1, \eps} \le -\frac{1}{2} \quad \text{and} \quad r_{2, \eps} > 0.
\]
Since $Y(t) \to 0$ as $t \to +\infty$, applying the maximum principle one deduces that
\[
Y(t) \le C (Z_{1, \eps}(t) + \e^{-t}).
\]
In particular,
\[
Y(t) \le C \e^{-\frac{t}{2}}.
\]
This gives the estimate for $v_2$.
\end{proof}

\begin{proof}[Proof of Theorem~\ref{prop4.1}]
By Lemma~\ref{lemma4.1} above, we have 
$$
\|v(t, \cdot)\|_{L^2} \leq C \, t^{- \frac{N - 1}{2}} \quad \forall t > 0.
$$
Inserting this estimate into estimate \eqref{5.3} for some $p > \frac{N}{2}$ the result follows.
\end{proof}

\begin{proof}[Proof of Theorem~\ref{thm1.2a}]
By Theorem~\ref{prop4.1}, the function $w : [0, +\infty) \to \RR$ given by
\[
w(t, \sigma) = t^{\frac{N-1}{2}} v(t, \sigma)
\]
is bounded. By a straightforward computation, $w$ satisfies
\begin{multline}\label{t2}
w_{tt} + \left(N - \tfrac{N-1}{t} \right) w_{t} +\left(N-1+\tfrac{N^2-1}{4t^2}\right)w +\Gd'w + \\
+ \myfrac{1}{t}\left(w^{q_{1}} - \tfrac{N(N-1)}{2} w \right) = t^{\frac{N-1}{2}}H,
\end{multline}
where $H$ is given by \eqref{new2}. Let $\phi : S_+^{N-1} \to \RR$ be the function defined by $\phi(\sigma) = \frac{\sigma_N}{|\sigma|}$; we recall that $\phi$ is an eigenfunction of $-\Delta'$ in $W_0^{1,2}(S_+^{N-1})$ associated to the first eigenvalue $N-1$. Let 
 $$
z(t)=\int\limits_{S^{N-1}_{+}} w(t,\gs) \gf(\gs) \, d\gs \quad \forall t \ge 0.
$$
Multiplying \eqref{t2} by $\phi$ and integrating over $S_+^{N-1}$, we obtain the following equation satisfied by $z$:
\[
z_{tt} +\left(N-\tfrac{N-1}{t}\right) z_t + \tfrac{N^2 - 1}{4 t^2} z + \frac{1}{t} \int\limits_{S_+^{N-1}} w^{q_1} \phi \, d\sigma - \tfrac{N(N-1)}{2t} z = t^{\frac{N-1}{2}} \int\limits_{S_+^{N-1}} H\phi \, d\sigma.
\]
Thus,
\begin{equation*}
z_{tt} +\left(N-\tfrac{N-1}{t}\right) z_t
+\frac{1}{t}\left(\theta z^{q_{1}}- \tfrac{N(N-1)}{2} z\right)=\Psi,
\end{equation*}
where
$$
\theta = \int\limits_{S^{N-1}_{+}} \gf^{q_{1}+1} \, d\gs
$$
and
$$
\Psi=t^{\frac{N-1}{2} }\int\limits_{S^{N-1}_{+}}{} H\gf \, d\gs-\tfrac{N^2-1}{4t^2} z + \frac{1}{t} \int\limits_{S^{N-1}_{+}} \big[ (z\gf)^{q_{1}} -  w^{q_{1}}\big]\gf \, d\gs.
$$
By Lemma~\ref{lemma4.1}, we have
\begin{equation}\label{7.1x}
\| z(t) \phi - w(t, \cdot)\|_{L^2} \leq C \, t^{\frac{N-1}{2}} \e^{- \frac{t}{2}}.
\end{equation}
Since
\begin{equation*}
\abs{(z \gf)^{q_{1}} - w^{q_{1}} } \leq q_{1}|z\phi - w | \big[ (z\gf)^{q_{1}-1} + w^{q_1 -1} \big],
\end{equation*}
$z$ is bounded in $\RR_+$ and $w$ is bounded in $\RR_+ \times S_+^{N-1}$,
\[
\begin{split}
\int\limits_{S^{N-1}_{+}} \big| (z\gf)^{q_{1}} -  w^{q_{1}}\big|\gf \, d\gs 
& \le \|z\phi - w\|_{L^2} \, \Big[	z^{q_1 - 1} \|\phi^{q_1 -1}\|_{L^2} + \|w^{q_1 -1}\|_{L^2} \Big]\\
& \le C \, t^{\frac{N-1}{2}} \e^{- \frac{t}{2}}.
\end{split}
\]
By Proposition~\ref{prop4.1}, \eqref{4.3} and Proposition~\ref{prop5.1} with $p=2$, 
\[
\|H(t, \cdot)\|_{L^2} \leq C \, t^{- \frac{N-1}{2}} \e^{-t}.
\]
Thus,
$$
\norm{\Psi(t,.)}_{L^\infty}\leq C \Big(\e^{-t} + t^{-2} + t^{\frac{N-3}{2}} \e^{- \frac{t}{2}}  \Big) \le \tilde C \, t^{-2}.
$$ 
By a straightforward modification of the end of the proof of \cite[Corollary~4.2]{BidRao:96}, $z$ admits a limit $\kappa\geq 0$ when $t\to +\infty$, where $\kappa$ satisfies
\begin{equation*}
\theta\kappa^{q_1}-\tfrac{N(N-1)}{2}\kappa=0.
\end{equation*}
Therefore, either $\kappa = 0$ or $\kappa = \left(\tfrac{N(N-1)}{2\theta}\right)^{\frac{N-1}{2}}$.\\
By \eqref{7.1x} we deduce that, as $t \to +\infty$, 
$$
t^{\frac{N-1}{2}} v(t,\cdot) \to \kappa \phi \quad \text{in } L^2(S^{N-1}_+).
$$
By Proposition~\ref{prop5.1} with $p > \frac{N}{2}$ and Morrey's estimates, we conclude that 
$$
t^{\frac{N-1}{2}} v(t,\cdot) \to \kappa \phi \quad \text{uniformly in } S^{N-1}_+.
$$
Rewriting the convergence in terms of $u$, we conclude that either \eqref{E6a} holds or
\begin{equation}\label{7a.1}
|x|^{N-1} \big(\log{\tfrac{1}{|x|}} \big)^{\frac{N-1}{2}} u(x) \to 0 \text{as $x \to 0$.}
\end{equation}
If \eqref{7a.1} holds, then $u$ must be continuous in view of Theorem~\ref{prop6.2}.
\end{proof}


\appendix
\section{Some \textsc{ode} lemmas}\label{sec2}

We gather in this section a couple of \textsc{ode} results which are used in this paper. These results are presumably well-known to specialists:

\begin{lemma}\label{lemma2.1}
Given $T > 0$, let $y \in C^2([T, +\infty))$ be a nonnegative function such that
\begin{equation*}
\left\{
\begin{aligned}
& y_{tt} + a y_t + by^q \leq c\, \e^{-t} \quad \text{in } (T, +\infty),\\
& \lim_{t \to +\infty}{y(t)} = 0,
\end{aligned}
\right.
\end{equation*}
where $q, a > 1$ and $b, c > 0$. Then, there exists $C > 0$ such that
\begin{equation}\label{2.2}
0 \leq y(t) \leq C \, t^{-\frac{1}{q-1}} \quad \forall t \geq T.
\end{equation}
\end{lemma}

\begin{proof}
Given $A > 0$, let
$$
z(t) = y(t) + A \e^{-t} \quad \forall t \geq T.
$$
Then, $z$ satisfies
$$
z_{tt} + a z_t + b z^q \leq \big[ c - (a-1) A \big] \e^{-t} + b (z^q - y^q).
$$
By convexity of the function $t  \in \RR_+ \mapsto t^q$, 
$$
y^q \geq z^q - q z^{q-1} A \e^{-t}.
$$
Thus,
\begin{equation}\label{2.3}
z_{tt} + a z_t + b z^q \leq \big[ c - (a-1 + bqz^{q-1}) A \big] \e^{-t}.
\end{equation}
Since $a > 1$ and $z(t) \to 0$ as $t \to \infty$, we can choose $T_1 > T$ and $A > 0$ sufficiently large so that the right-hand side of \eqref{2.3} is negative on $[T_1, \infty)$. Thus,
\begin{equation}\label{2.4}
z_{tt} + a z_t + b z^q \leq 0 \quad \text{in } [T_1, \infty).
\end{equation}
Let $w = z^{1-q}$. By a straightforward computation, we have
\begin{equation}\label{2.5}
w_{tt} + a w_t \geq - (q-1) \frac{z_{tt} + a z_t}{z^q}. 
\end{equation}
Combining \eqref{2.4}--\eqref{2.5}, we deduce that
$$
w_{tt} + a w_t \geq b(q-1) \quad \text{in } [T_1, \infty).
$$
The function $x = w_t$ satisfies
$$
x_t + a x \geq b(q-1) \quad \text{in } [T_1, \infty).
$$
Thus, taking $T_2 > T_1$ sufficiently large, 
$$
x(t) \geq \tfrac{b(q-1)}{a} + c_1 \e^{-at} \geq \tfrac{b(q-1)}{2 a} \quad \forall t \geq T_2.
$$
Since $w_t = x$, choosing $T_3 > T_2$ large enough, we then get
$$
w(t) \geq \tfrac{b(q-1)}{4 a} \, t \quad \forall t \geq T_3.
$$
Therefore,
$$
z(t) \leq \left(\tfrac{4 a}{b(q-1)} \, t^{-1}\right)^{\frac{1}{q-1}}  
\quad \forall t \geq T_3.
$$
We can now enlarge the constant in the right-hand side so that this estimate holds for every $t \geq T$.
This immediately implies \eqref{2.2}.
\end{proof}

\begin{lemma}\label{lemma2.3}
Let $a, a_1, b, b_1 \in \RR$ with $a \ne 0$. Then, the equation
\begin{equation*}
y_{tt} + \big( a - \tfrac{a_1}{t} \big) y_{t} + \tfrac{1}{t} \big( b - \tfrac{b_1}{t} \big) y = 0 \quad \text{in } (0, +\infty),
\end{equation*}
has two linearly independent solutions $y_1$ and $y_2$ such that
\begin{equation*}
y_1(t) = t^{a_1 + \frac{b}{a}} \e^{-at} (1 + o(1)) \quad \text{and} \quad y_2(t) = t^{-\frac{b}{a}} (1 + o(1))
\end{equation*}
for $t$ sufficiently large.
\end{lemma}

\begin{proof}
Let 
$$
z(t) =  \e^{\frac{at}{2}} \, t^{- \frac{a_1}{2}} y(t).
$$
Then, $z$ satisfies the equation
\begin{equation*}
z_{tt} - \big( \tfrac{a^2}{4} - \tfrac{A_1}{t} + \tfrac{A_2}{t^2} \big) z = 0,
\end{equation*}
where $A_1 = b + \frac{aa_1}{2}$ and $A_2 = b_1 + \frac{a_1}{2} + \frac{a_1^2}{4}$.
By \cite[pp.~126--127]{Bel:53}, the equation satisfied by $z$ has two linearly independent solutions with the following asymptotic behaviors as $t \to +\infty$:
$$
z_1(t) = \e^{-\frac{at}{2}} \, t^{ \frac{A_1}{a}} ( 1 + o(1)) \quad \text{and} \quad z_2(t) = \e^{\frac{at}{2}} \, t^{- \frac{A_1}{a}} ( 1 + o(1) ).
$$
Rewriting these formulas in terms of the function $y$, the result follows.
\end{proof}






\end{document}